\documentclass[12pt]{article}
 \usepackage[margin=1in]{geometry} 
\usepackage{amsmath,amsthm,amssymb,amsfonts}
 \usepackage{amssymb}

\usepackage[pagewise]{lineno}
\usepackage[utf8]{inputenc}

\usepackage{hyperref}
 \usepackage{csquotes}
\usepackage[utf8]{inputenc}
\usepackage[english]{babel}
\usepackage{xcolor}

\usepackage{enumitem}

\numberwithin{equation}{section}
\newtheorem{theorem}{Theorem}[section]
\newtheorem{lemma}[theorem]{Lemma}
\newtheorem{remark}{Remark}[section]

\providecommand{\keywords}[1]
{
  \small	
  \textbf{\textit{Keywords:}} #1
}

\newcommand{\MSC}[1]{%
  \small
  \textbf{\textit{Mathematics Subject Classification:}} #1
}
\title{A Sharp Global Boundedness Result for Keller–Segel–(Navier–)Stokes Systems with Rapid Diffusion and Saturated Sensitivities  }
\author{ 
      Minh Le\thanks{Institute for Theoretical Sciences, Westlake University,  China \texttt{(leminh@westlake.edu.cn)}} 
    }
\date{}

\begin{document}
\maketitle

\begin{abstract}
We investigate the Keller--Segel--(Navier--)Stokes system posed in a smooth bounded domain
\(\Omega \subset \mathbb{R}^N\) with \(N = 2,3\):
\begin{equation*}
\begin{cases}
n_t + u \cdot \nabla n = \Delta n - \nabla \cdot \big( n S(n)\nabla c \big), \\[2mm]
u \cdot \nabla c = \Delta c - c + n, \\[2mm]
u_t + \kappa (u \cdot \nabla) u = \Delta u - \nabla P + n \nabla \phi, \\[2mm]
\nabla \cdot u = 0,
\end{cases}
\end{equation*}
where \(\kappa \in \left \{0,1 \right \} \), the given gravitational potential \(\phi \in W^{2, \infty}(\Omega)\), and the
chemotactic sensitivity function \(S \in C^2([0,\infty))\).

Under no-flux boundary conditions for \(n\) and \(c\), together with the Dirichlet boundary
condition for \(u\), we show that, provided the initial data satisfy suitable regularity
assumptions, the following results hold:
\begin{itemize}
    \item If \(N = 2\), \(\kappa = 1\), and the sensitivity function satisfies
    \(\lim_{\xi \to \infty} S(\xi) = 0\), then the Keller--Segel--Navier--Stokes system admits
    a global classical solution that remains uniformly bounded in time.

    \item If \(N = 3\), \(\kappa = 0\), and \(S\) satisfies
    \[
    |S(\xi)| \le K_S (\xi + 1)^{-\alpha} \quad \text{for all } \xi \ge 0,
    \]
    with some constants \(K_S > 0\) and \(\alpha > \frac{1}{3}\), then the
    Keller--Segel--Stokes system possesses a global bounded classical solution.
\end{itemize}
Our results are optimal, since it is well established that, in the absence of
fluid effects, blow-up can occur when
$S \equiv \mathrm{const}$ in two dimensions, or when $\alpha < \tfrac{1}{3}$
in three dimensions.

\end{abstract}
\\
\MSC{35B35, 35K45, 35K55, 92C15, 92C17}

\numberwithin{equation}{section}

\newtheorem{Corollary}{Corollary}[theorem]
\allowdisplaybreaks

\section{Introduction}
Chemotaxis, which describes the directed movement of cells or microorganisms in response to chemical stimuli, has been intensively studied from a mathematical perspective since the seminal model was introduced in the 1970s \cite{Keller}
\begin{equation} \label{KS}
    \begin{cases}
        n_t = \Delta n - \nabla \cdot (n \nabla c), \\
        \gamma c_t = \Delta c - c + n,
    \end{cases}
\end{equation}
where $\gamma \in \{0,1\}$.  One of the most striking features of this system, which has attracted considerable attention from the mathematical community, is the {critical mass phenomenon}. In two spatial dimensions, if the initial mass of the cell density lies below a certain threshold, solutions remain global and bounded \cite{Dolbeault, Dolbeault1, NSY}. In contrast, when this mass exceeds the threshold, solutions may blow up in finite time \cite{Nagai1, Nagai2, Nagai3, Nagai4}. This phenomenon, however, does not persist in higher-dimensional settings. Indeed, it was shown in \cite{Winkler-2010, Winkler-2013} that in dimensions $N \geq 2$, there exist initial data with arbitrarily small mass that nevertheless lead to blow-up.

Let us now shift our attention to a variation of the system~\eqref{KS}, given by
\begin{equation} \label{KS'}
    \begin{cases}
        n_t = \Delta n - \nabla \cdot \bigl(n(1+n)^{-\alpha} \nabla c\bigr), \\
        \gamma c_t = \Delta c - c + n,
    \end{cases}
\end{equation}
where $\alpha \in \mathbb{R}$. This model also exhibits a criticality phenomenon; however, in contrast to \eqref{KS}, the critical parameter is not the total mass but the exponent $\alpha$. 
More precisely, it has been shown that the value
\[
\alpha^* = 1 - \frac{2}{N}
\]
is critical in the following sense: if $\alpha > \alpha^*$, then solutions to system~\eqref{KS'} are global and bounded \cite{Winkler-Horstmann, Winkler-2011}, whereas if $\alpha < \alpha^*$, there exist solutions that become unbounded when $\Omega$ is a ball in $\mathbb{R}^N$ with $N \geq 2$ \cite{CieslakStinner2015, Winkler2010VolumeFilling}.

When microorganisms or cells inhabit a fluid environment, the mathematical modeling of chemotactic phenomena becomes substantially more intricate due to their interaction with the surrounding fluid. Such interactions may involve buoyancy effects, gravitational forces, and transport mechanisms induced by the fluid flow. Motivated by the seminal work of Tuval et al.~\cite{TuvalEtAl2005}, we consider a chemotaxis--fluid system posed in a bounded domain $\Omega \subset \mathbb{R}^N$, where $N \in \{2,3\}$:

\begin{equation}  \label{1} 
    \begin{cases}
        n_t + u \cdot \nabla n = \Delta n - \nabla \cdot \left( n S(n)\nabla c \right), \qquad &x\in \Omega, t>0,\\
       u\cdot \nabla c = \Delta c - c + n, \qquad &x\in \Omega, t>0,\\
        u_t + \kappa ( u \cdot \nabla ) u +\nabla P = \Delta u + n \nabla \phi ,\quad \nabla \cdot u =0,\qquad &x\in \Omega, t>0,
    \end{cases}
\end{equation}
 In this system, $n=n(x,t)$, $c=c(x,t)$, $u = u(x,t)$ and $P = P(x,t)$ denote the cell density. chemical concentration, the fluid velocity and pressure, respectively.
Here $S\in C^2 \left [0, \infty \right )$ such that 
\begin{equation}\label{S}
    \begin{cases}
        \lim_{\xi \to \infty}S(\xi) =0 ,  &\text{ if }N=2,\\
         |S(\xi)| \leq K_S (\xi+1)^{-\alpha} \qquad \text{for all }\xi \geq 0,  &\text{ if }N=3,
    \end{cases}
\end{equation}
for some $K_S>0$ and $\alpha>0$, and the given gravitational potential function $\phi$ fulfills 
\begin{align} \label{phi}
    \phi \in W^{2, \infty}(\Omega).
\end{align}
 The system \eqref{1} satisfies the Neumann-Neumann-Dirichlet boundary conditions
\begin{equation} \label{bdry}
    \frac{\partial n}{\partial \nu} =0, \quad \frac{\partial c}{\partial \nu} =0, \quad \text{and } u=0 ,\qquad x \in \partial \Omega, t>0,
\end{equation}
and initial conditions 
\begin{align*}
    n(x,0) =n_0(x), \qquad u(x,0) =u_0(x) \qquad \text{for all } x \in \Omega,
\end{align*}
with 
\begin{equation} \label{initial}
    \begin{cases}
        n_0 \in C^{0} (\bar{\Omega}) &\text{ with } n_0\geq 0   \\
      u_0 \in D(A^\beta) &\text{with } \beta \in \left ( \frac{3}{4}, 1\right ) \text{when }n=3 \quad \text{and }  \beta \in \left ( \frac{1}{2}, 1\right ) \text{when }n=2,
    \end{cases}
\end{equation}
where $A := -\mathcal{P}\Delta$ denotes the Stokes operator on 
\[
L^2_\sigma(\Omega) := \left\{ \phi \in L^2(\Omega; \mathbb{R}^N) \,\middle|\, \nabla \cdot \phi = 0 \right\},
\]
with $\mathcal{P}$ representing the Helmholtz projection from $L^2(\Omega;\mathbb{R}^N)$ into $L^2_\sigma(\Omega)$.

Let us briefly review several results in the literature concerning system~\eqref{1}
in the two-dimensional setting.

Assuming that
\[
|S(\xi)| \le C(\xi+1)^{-\alpha} \quad \text{for all } \xi \ge 0,
\]
with some constants $C>0$ and $\alpha>0$, the existence of globally bounded
solutions to the simplified fully parabolic Keller--Segel--Stokes system (with
$\kappa = 0$) was established in \cite{WangXiang2015}. Closely related results on
global existence and boundedness for the fully parabolic
Keller--Segel--Navier--Stokes system~\eqref{1} were later obtained in
\cite{WangWinklerXiang2018} and \cite{LiZheng2022}. In the particular case $S \equiv 1$, it was shown in \cite{Winkler_SIAM} that if
\(
\int_\Omega n_0 < 2\pi,
\)
then system~\eqref{1} admits a global generalized solution which eventually
becomes smooth and converges to a constant steady state. By contrast, if
\(
\int_\Omega n_0 \in (4\pi, \infty)\setminus \{4k\pi \mid k \in \mathbb{N}\},
\)
unbounded solutions were shown to exist for $u_0 \equiv 0$ and
$\nabla \phi \equiv 0$; see \cite{HorstmannWang2001_blowup}. In order to circumvent
this critical-mass phenomenon, the authors in \cite{WT2023} proved that, under
the optimal condition
\[
\lim_{\xi \to \infty} S(\xi) = 0,
\]
the fully parabolic chemotaxis--Navier--Stokes system admits globally bounded
solutions.

In contrast, research on the parabolic--elliptic case remains relatively scarce. For instance, it was shown in \cite{Winkler+2020} that system~\eqref{1}
possesses global bounded solutions when the chemotactic sensitivity is repulsive
($S \equiv -1$). Furthermore, \cite{Zheng2021} established the existence of
global classical bounded solutions for the simplified Keller--Segel--Stokes
system under the condition
\[
|S(\xi)| \le C(\xi+1)^{-\alpha} \quad \text{for all } \xi \ge 0,
\]
with some $C>0$ and $\alpha>0$. In the special case where $S \equiv 1$,
$\phi = c$, and $\Omega = \mathbb{R}^2$, it was demonstrated in \cite{GongHe2021}
that solutions to the parabolic--elliptic system exist globally in time provided
that \(
\int_\Omega n_0 < 8\pi.
\)
For additional results on global solvability in a wide range of variants of
system~\eqref{1}, we refer the interested reader to
\cite{Cao_2016, Kozono_2016, Duan_2010, Wang2023, DaiLiu}.
These aforementioned findings naturally motivate the investigation of sharp conditions on the
sensitivity function $S$ that guarantee global existence and boundedness of
solutions to the parabolic--elliptic system~\eqref{1}. 

Accordingly, the first main objective of the present work is to establish global
solvability and boundedness for system~\eqref{1}, in close analogy with the fully
parabolic framework studied in \cite{WT2023}. More precisely, we formulate the
following theorem.

\begin{theorem} \label{thm1}
    Let $\kappa=0$, $\Omega \subset \mathbb{R}^2$ be a bounded domain with smooth boundary and assume that \eqref{S} and \eqref{phi} hold. Then for any choice of $n_0$ and $u_0$ fulfilling \eqref{initial}, the problem \eqref{1} under the boundary conditions \eqref{bdry} possesses a global classical solution $(u,c,u,P)$ satisfying 
    \begin{equation*}
        \begin{cases}
            n \in C^0 \left (\bar{\Omega} \times [0,\infty) \right ) \cap C^{2,1} \left ( \bar{\Omega} \times (0,\infty) \right ), \\
            c \in C^{2,0}\left ( \bar{\Omega} \times (0,\infty) \right ) , \\
            u \in C^0 \left ([0,\infty); D(A^\beta)  \right ) \cap C^{2,1}\left ( \bar{\Omega} \times (0,\infty); \mathbb{R}^2 \right ) \quad \text{and }\\
            P \in C^{1,0} \left ( \bar{\Omega}\times (0,\infty) \right ),
        \end{cases}
    \end{equation*}
    as well as $n$ and $c$ are positive in $\bar{\Omega}\times (0, \infty)$. Moreover, the solution is globally bounded in the sense that 
    \begin{align} \label{thm1-1}
    \sup_{t>0} \left \{    \left \| n(\cdot,t) \right \|_{L^\infty(\Omega)} +    \left \| c(\cdot,t) \right \|_{W^{1,\infty}(\Omega)} +  \left \| u(\cdot,t) \right \|_{L^\infty(\Omega)} \right \} < \infty. 
    \end{align}
\end{theorem}

\begin{remark}
   Theorem \ref{thm1} extends and strengthens the results obtained in
\cite{Zheng2021}, where comparable boundedness results were established only for the
two-dimensional Keller--Segel--Stokes system. 
\end{remark}

The proof of Theorem~\ref{thm1} follows a well-established strategy commonly
employed in chemotaxis--fluid systems. The approach consists of several
steps: first, one derives a uniform-in-time bound for $\int_\Omega |u|^2$;
second, a corresponding bound for $\int_\Omega n \log(n+1)$ is obtained. These
estimates then allow us to establish $L^p$ bounds for both $u$ and $n$, which
can subsequently be combined with a Moser iteration procedure to deduce an
$L^\infty$ bound for $n$.

However, there are several challenges inherent in this approach. The first
difficulty lies in obtaining a uniform-in-time $L^2$ bound for $u$, which
requires a priori control of the quantity
\[
\int_t^{t+\tau} \int_\Omega n \log(n+1), \qquad \text{where } 
\tau = \min \Bigl\{1, \frac{T_{\rm max}}{2} \Bigr\}, 
\]
and $T_{\rm max} \in (0, \infty]$ denotes the maximal existence time
specified in Lemma~\ref{Local}. To overcome this obstacle, we are inspired by
the method developed in \cite{WT2023}, which involves considering the functional
\[
F(t) = - \int_\Omega  \log(n+B), \qquad t>0,
\]
for some $B>0$, and applying Moser--Trudinger type inequalities (see
Lemmas~\ref{Moser-Trudinger1} and~\ref{Moser-Trudinger2}).

The second difficulty arises when establishing a uniform-in-time $L \log L$
bound for $n$. Direct substitution of the second equation into the functional
\[
G(t) = \int_\Omega n \log n, \qquad t>0,
\]
leads to the term $\int_\Omega n S(n) u \cdot \nabla c$, which cannot be
directly controlled. Remarkably, an indirect control using $\Delta c$ (see Lemma \ref{LM}) together
with Young's inequality allows us to overcome this obstacle.

Let us turn our attention to some existing results in the literature in the three-dimensional setting ($N=3$).

In \cite{YZ2016}, the authors demonstrated that the fully parabolic Keller--Segel--Stokes system
admits a globally bounded classical solution in three dimensions provided that
$\alpha > \tfrac{1}{2}$. When compared with the corresponding fluid-free case,
this result is not sharp, since it is known that blow-up may occur when
$\alpha < \tfrac{1}{3}$. Subsequently, it was demonstrated in
\cite{Winkler-2018} that the inclusion of fluid interaction does not influence
the global solvability and boundedness of solutions. More precisely, the author
proved that the condition $\alpha > \tfrac{1}{3}$ is sufficient to guarantee the
existence of global bounded solutions, thereby closing this gap in the existing
theory.

It is therefore natural to raise the following related question that whether the
condition $\alpha > \tfrac{1}{3}$ can also prevent blow-up for the
parabolic--elliptic system~\eqref{1}. In \cite{PJ2022}, the authors partially
addressed this issue by showing that the stronger condition $\alpha >
\tfrac{1}{2}$ ensures the global existence and boundedness of solutions to
system~\eqref{1}. The second objective of the present paper is to provide a
complete answer to this question. More precisely, we establish the global
existence and boundedness of solutions to system~\eqref{1} under the optimal
assumption $\alpha > \tfrac{1}{3}$, as stated in the following theorem.

\begin{theorem} \label{thm2}
  Let $\kappa = 0$, and let $\Omega \subset \mathbb{R}^3$ be a bounded domain with smooth boundary. Assume that \eqref{phi} holds and that $S \in C^2([0,\infty))$ satisfies condition~\eqref{S} with $\alpha > \frac{1}{3}$. Then, for any initial data $n_0$ and $u_0$ fulfilling \eqref{initial}, the problem \eqref{1} under the boundary conditions \eqref{bdry} possesses a global classical solution $(n, c, u, P)$ satisfying
\begin{equation*}
    \begin{cases}
        n \in C^0(\bar{\Omega} \times [0,\infty)) \cap C^{2,1}(\bar{\Omega} \times (0,\infty)), \\
        c \in C^{2,0}(\bar{\Omega} \times (0,\infty)), \\
        u \in C^0([0,\infty); D(A^\beta)) \cap C^{2,1}(\bar{\Omega} \times (0,\infty); \mathbb{R}^3), \\
        P \in C^{1,0}(\bar{\Omega} \times (0,\infty)),
    \end{cases}
\end{equation*}
with $n$ and $c$ remaining positive in $\bar{\Omega} \times (0, \infty)$. Moreover, there exists a constant $C > 0$ such that
\begin{align} \label{thm2-1}
    \| n(\cdot,t) \|_{L^\infty(\Omega)} + \| c(\cdot,t) \|_{W^{1,\infty}(\Omega)} + \| u(\cdot,t) \|_{L^\infty(\Omega)} \leq C \quad \text{for all } t > 0.
\end{align}

\end{theorem}
\begin{remark}
    It remains unknown that if Theorem \ref{thm2} still holds under a weaker assumption for positive chemical sensitivity function $S$ that $\lim_{\xi \to  \infty} \xi^{\frac{1}{3}}S(\xi) =0$.
\end{remark}

The proof of Theorem~\ref{thm2}, though unexpectedly straightforward, proceeds along a classical line of reasoning by introducing the functional
\begin{align*}
    H(t) = \int_\Omega n^p, \qquad t>0.
\end{align*}
The principal challenge arises from controlling the term involving the velocity field $u$, namely
\[
    \int_\Omega n^{p-2\alpha} |u|^2 c^2.
\]
By virtue of the elliptic regularity result established in Lemma~\ref{L4}, one may heuristically approximate $c$ by $n^{\frac{1}{3}}$ and use the fact that $u$ is bounded in $L^p$ with $p<3$ (see Lemma \ref{L2}). Under this approximation, the above term can be dominated by
\[
   \left ( \int_\Omega n^{3p-6\alpha+2} \right )^{\frac{1}{3}},
\]
which can be successfully absorbed into the diffusion term when $\alpha>\frac{1}{3}$.\\

The paper is organized as follows. In Section~\ref{S2}, we recall the well-established local well-posedness theory for system~\eqref{1} and introduce several useful inequalities that will be employed in subsequent sections. Section~\ref{S3} is devoted to the proof of Theorem~\ref{thm1}, which establishes global existence in two spatial dimensions. Finally, in Section~\ref{S4}, we prove Theorem~\ref{thm2}, thereby obtaining the global existence result in three spatial dimensions.

\section{Preliminaries} \label{S2}
In this section, we establish the local existence of solutions to system
\eqref{1} and introduce several useful inequalities, including differential inequalities and Moser--Trudinger type inequalities in two dimensions. We begin by recording
the following basic local well-posedness result. Since the proof follows from
a standard fixed-point argument, similar to those in
\cite{Winkler2012}[Lemma~2.1] and \cite{Winkler+2020}[Lemma~2.2], we omit the
details to avoid redundancy.

\begin{lemma} \label{Local}
    Let $\Omega \subset \mathbb{R}^N$ with $N=2,3$ be a bounded domain with smooth boundary. Suppose that $\kappa=0$ when $N=3$ and $\kappa=1$ when $N=2$ and \eqref{S} and \eqref{phi} hold. Then there exists a maximal existence time $T_{\rm max}\in (0, \infty]$ and a unique quadruple $(n,c,u,P)$ of functions 
    \begin{equation}
        \begin{cases}
            n \in C^0 \left (\bar{\Omega} \times [0,T_{\rm max}) \right ) \cap C^{2,1} \left ( \bar{\Omega} \times (0,T_{\rm max}) \right ), \\
            c \in C^{2,0}\left ( \bar{\Omega} \times (0,T_{\rm max}) \right ) , \\
            u \in C^0 \left ([0,T_{\rm max}); D(A^\beta)  \right ) \cap C^{2,1}\left ( \bar{\Omega} \times (0,T_{\rm max}); \mathbb{R}^N \right ) \quad \text{and }\\
            P \in C^{1,0} \left ( \bar{\Omega}\times (0,T_{\rm max}) \right ),
        \end{cases}
    \end{equation}
    which are such that $n>0$ and $c>0$ in $\bar{\Omega}\times (0,T_{\rm max})$, that $(n,c,u,P)$ solves \eqref{1} under the boundary conditions \eqref{bdry} and initial conditions \eqref{initial} in the classical sense, and that 
    \begin{align}\label{Local.1}
        \text{if }T_{\rm max}< \infty \quad \text{then } \limsup_{t\to T_{\rm max}} \left \{ \left \|n(\cdot,t) \right \|_{L^\infty(\Omega)}+\left \|c(\cdot,t) \right \|_{W^{1,\infty}(\Omega)}+ \left \|A^\beta u(\cdot,t) \right \|_{L^2(\Omega)}  \right \} = \infty.
    \end{align}
    Moreover, 
     \begin{align} \label{Local.2}
        \int_\Omega n(\cdot,t) = \int_\Omega c(\cdot,t) = \int_\Omega n_0 =m \qquad \text{for all }t\in (0,T_{\rm max}).
    \end{align}
\end{lemma}
From now on, we denote $(n,c,u,P)$ as a classical solution in $\Omega \times (0,T_{\rm max})$ as established in Lemma \ref{Local}. Let us now introduce the following elementary differential inequality which will later be applied in Lemma \ref{L2u} and Lemma \ref{Gradu}. For a detailed proof, we refer the reader to \cite[ Lemma 3.4]{Winkler+2019}
\begin{lemma} \label{ODI}
    Let $T\in (0, \infty]$ and $\tau \in (0,T)$, $a>0$ and $b>0$, and assume that $y: [t_0,T) \to [0,\infty)$ for some $t_0 \in \mathbb{R}$ is absolutely continuous and such that 
    \begin{align*}
        y'(t)+ay(t) \leq h(t)\qquad \text{for a.e } t \in (t_0,T),
    \end{align*}
    with some nonnegative function $h \in L^1_{loc}([t_0,T))$ satisfying 
    \begin{align*}
        \frac{1}{\tau }\int_t^{t+\tau }h(s)\, ds \leq b \qquad \text{for all }t\in [t_0,T-\tau ). 
    \end{align*}
    Then 
    \begin{align*}
        y(t) \leq y(t_0)e^{-a(t-t_0)}+\frac{b \tau }{1- e^{-a\tau }} \qquad \text{for all }t\in [0,T).
    \end{align*}
\end{lemma}

For the later use in Lemma \ref{LlolL}, we need the following lemma.

\begin{lemma} \label{ODI'}
   Let  $a\in \mathbb{R}$ and $T \in (a, \infty]$. Assume that $y: [a,T) \to [0,\infty)$  is absolutely continuous and such that 
    \begin{align*}
         \int_t ^{t+\tau} y(s) \,ds  \leq L_1\qquad \text{for all }t\in (a,T-\tau),
    \end{align*}
    where  $\tau= \min \left \{1, \frac{T-a}{2} \right \}$ and $L_1>0$, and 
    \begin{align*}
         y'(t) \leq h(t)y(t)+g(t) \qquad \text{for a.e } t \in (a,T),
    \end{align*}
    where $h$ and $g $ are nonnegative continuous functions in $[a,T)$ such that 
    \begin{align*}
        \int_t ^{t+\tau} h(s) \,ds  \leq L_2 \quad \text{and } \int_t ^{t+\tau} g(s) \,ds  \leq L_3 \qquad \text{for all }t\in (a,T-\tau),
    \end{align*}
    where $L_2>0$, and $L_3>0$. Then
    \begin{align*}
        y(t) \leq \frac{L_1 e^{L_2}}{\tau}+ L_3 e^{L_2} \qquad \text{for all }t\in (0,T).
    \end{align*}
    
\end{lemma}
\begin{proof}
    The assumption $ \sup_{t \in (a, T-\tau)}\int_t ^{t+\tau }y(s)\, ds \leq L_1$, implies that for any $t \in (a, T)$ then there exists $t_0 \in (t-\tau , t)\cap [a,\infty)$ depending on $t$ such that 
  \begin{align*}
      y(t_0) \leq \frac{L_1}{\tau}.
  \end{align*}
  Multiplying $e^{- \int_{t_0}^t h(s)\, ds}$ to the inequality $y'(t) \leq h(t)y(t)+g(t)$ and integrating from $t_0$ to $t$, we deduce that 
  \begin{align}
      y(t) &\leq y(t_0)e^{  \int_{t_0}^t h(s)\, ds} + \int_{t_0}^t e^{\int_{s}^t h(z)\, dz} g(s)\, ds \notag \\
      &\leq \frac{L_1 e^{L_2}}{\tau}+ L_3 e^{L_2} \qquad \text{for all } t\in (a,T),
  \end{align}
  which completes the proof.
\end{proof}
Next, we recall from \cite[Lemma 2.2]{Winkler_SIAM} a variant of the Moser--Trudinger inequality, which will later be applied to establish the $L^2$ boundedness of $u$ in Lemma~\ref{L2u}.

\begin{lemma} \label{Moser-Trudinger1}
    Let $\Omega \subset \mathbb{R}^2$ be a bounded domain with smooth boundary. Then for all $\epsilon>0$ there exists $M=M(\epsilon,\Omega)>0$ such that if $0  \not\equiv \phi \in  C^0(\bar{\Omega}) $ is nonnegative and $\psi \in W^{1,2}(\Omega)$, then for each $a>0$
    \begin{align}
        \int_\Omega \phi |\psi| \leq \frac{1}{a}\int_\Omega \phi \ln \frac{\phi}{\bar{\phi}} + \frac{(1+\epsilon)a}{8\pi}\cdot \left \{ \int_\Omega \phi \right \} \cdot \int_\Omega |\nabla \psi|^2+Ma \cdot \left \{ \int_\Omega \phi \right \} \cdot \left \{ \int_\Omega |\psi| \right \}^2+ \frac{M}{a}\int_\Omega \phi,
    \end{align}
    where $\bar{\phi}:= \frac{1}{|\Omega|}\int_\Omega \phi $.
\end{lemma}
As a consequence of Lemma \ref{Moser-Trudinger1}, we can derive the following inequality that will later be used in Lemma \ref{L2u}. For a detailed proof, we refer the reader to Lemma 2.2 in \cite{WT2023}.

\begin{lemma} \label{Moser-Trudinger2}
    Let $\Omega \subset \mathbb{R}^2$ be a bounded domain with smooth boundary. Then for any $\epsilon>0$ there exists $\Gamma=\Gamma(\epsilon,\Omega)>0$ such that then for all $a>0$ and $M\geq 1$ and for each $0  \not\equiv \phi \in  C^0(\bar{\Omega}) $ is nonnegative and $\psi \in W^{1,2}(\Omega)$,
    \begin{align}\label{Moser-Trudinger2-1}
        \int_\Omega \phi |\psi| &\leq \frac{1+\epsilon}{2\pi a} \cdot \left \{ \int_\Omega \phi \right \}\cdot \left \{ \int_\Omega \frac{|\nabla \phi|^2}{(\phi+M)^2} \right \} + \frac{(1+\epsilon)a}{8 \pi} \cdot \left \{ \int_\Omega \phi \right \}\cdot \left \{ \int_\Omega |\nabla \psi|^2 \right \} \notag \\
        &\quad+\Gamma a\cdot \left \{ \int_\Omega \phi \right \}\cdot \left \{ \int_\Omega |\psi| \right \}^2 +\frac{8\Gamma}{a} \cdot \left \{ \int_\Omega \phi \right \}^3+\frac{\Gamma(8|\Omega|^2M^2+2)}{a} \cdot \left \{ \int_\Omega \phi \right \}+\frac{2|\Omega|}{ea},
    \end{align}
    and
    \begin{align}\label{Moser-Trudinger2-2}
        \int_\Omega \phi \log( \phi+M) &\leq \frac{1+\epsilon}{2\pi}\cdot \left \{ \int_\Omega \phi \right \}\cdot \left \{ \int_\Omega \frac{|\nabla \phi|^2}{(\phi+M)^2} \right \} \notag \\
        &\quad+8 \Gamma\cdot \left \{ \int_\Omega \phi \right \}^3 +\Gamma (8M^2|\Omega|^2+1)\int_\Omega \phi + \frac{|\Omega |}{e} . 
    \end{align}
\end{lemma}

\section{Two-Dimensional Keller-Segel-Navier-Stokes systems } \label{S3}
In this section, we consider the case $N=2$ and assume that $S$ satisfies condition~\eqref{S}. First, in Subsection~\ref{S3-1}, we derive several useful estimates for solutions to system~\eqref{1}. Next, in Subsection~\ref{S3-2}, we establish a uniform-in-time $L\log L$ bound for $n$ and use it to obtain an $L^p$ bound for $n$. Finally, in Subsection~\ref{S3-3}, we prove an $L^\infty$ bound for the solutions.

\subsection{A Priori Estimates} \label{S3-1}

As a consequence of the decay condition \eqref{S} for $S$, we can derive the following elementary estimate which will later used in Lemma \ref{L1}.
\begin{lemma} \label{LTW}
    For any $\epsilon>0$, there exists $M_\epsilon>0$ such that whenever $M \geq M_\epsilon$,
    \begin{align}\label{LTW-1}
        \frac{\xi |S(\xi)|}{\xi+M} \leq \epsilon \qquad \text{for all }\xi \geq 0.
    \end{align}
\end{lemma}
\begin{proof}
    By the condition \eqref{S}, for any $\epsilon>0$, there exists $A_\epsilon>0$ such that 
    \begin{align}\label{LTW.1}
        |S(\xi)| \leq \epsilon \qquad \text{for all } \xi >  A_\epsilon.
    \end{align}
    One the other hand, the boundedness of $S$ due to \eqref{S} implies that 
    \begin{align}\label{LTW.2}
           \frac{\xi |S(\xi)|}{\xi+M}  \leq \frac{A_\epsilon \left \|S \right \|_{L^\infty((0, \infty))} }{M+B_\epsilon} \qquad\text{for all } \xi \leq  A_\epsilon.
    \end{align}
    Now, choosing $M_\epsilon := \max \left \{\frac{A_\epsilon (\left \|S \right \|_{L^\infty((0, \infty))}- \epsilon)}{\epsilon} , 1\right \}$ then for any $M\geq M_\epsilon$ we have
    \begin{align}\label{LTW.3}
        \frac{A_\epsilon \left \|S \right \|_{L^\infty((0, \infty))} }{M+B_\epsilon} \leq \epsilon.
    \end{align}
    Collecting \eqref{LTW.1}, \eqref{LTW.2}, and \eqref{LTW.3} implies \eqref{LTW-1}. The proof is now complete.
    
\end{proof}

The following lemma provides a simple yet crucial estimate, which will later be applied in Lemma~\ref{L1}.

\begin{lemma} \label{Lnc}
There exists $C>0$ such that 
\begin{align*}
    \int_\Omega nc \leq \frac{4\pi}{m} \int_\Omega n \log (n+1)+C \qquad \text{for all } t\in (0,T_{\rm max}),
\end{align*}
where $m=\int_\Omega n(\cdot,t)$.
\end{lemma}
\begin{proof}
    Applying Lemma \ref{Moser-Trudinger1} with $\epsilon=1$ and $a= \frac{2\pi}{m}$ and using the fact that
    \begin{align*}
        \int_\Omega |\nabla c|^2+\int_\Omega c^2 = \int_\Omega nc,
    \end{align*}
     there exists $c_1>0$ such that  
    \begin{align*}
        \int_\Omega nc &\leq \frac{2 \pi}{m} \int_\Omega n \log \frac{n}{\bar{n}} + \frac{1}{2}\int_\Omega |\nabla c|^2 +c_1 \notag \\
        &\leq \frac{2 \pi}{m} \int_\Omega n \log n + \frac{1}{2}\int_\Omega n c  - 2 \pi \log \left ( \frac{m}{|\Omega|} \right )  +c_1.
    \end{align*}
    This further entails that there exists $c_2>0$ satisfying 
    \begin{align*}
        \int_\Omega nc \leq \frac{4 \pi}{m} \int_\Omega n \log (n+1) +c_2 \qquad \text{for all } t\in (0,T_{\rm max}),
    \end{align*}
    which completes the proof.
\end{proof}
The following lemma, which provides spatio-temporal averaged bounds for the solutions, constitutes the first main ingredient to the proof of the first main result.

\begin{lemma} \label{L1}
    There exists $C>0$ such that 
    \begin{align}  \label{L1-1}
        \int_t^{t+\tau }\int_\Omega n \log (n+1) \leq C \qquad \text{for all }t\in (0,T_{\rm max} - \tau)
    \end{align}
    and 
    \begin{align} \label{L1-2}
        \int_t^{t+\tau }\int_\Omega \frac{|\nabla n|^2}{(n+1)^2} \leq C \qquad \text{for all }t\in (0,T_{\rm max} - \tau),
    \end{align}
    where $\tau = \min \left \{1, \frac{T_{\rm max}}{2} \right \}$.
\end{lemma}
\begin{proof}
    From Lemma \ref{LTW}, there exists $B>1$ such that 
    \begin{align} \label{L1.1}
        \frac{\xi |S(\xi)|}{\xi +B} \leq \frac{1}{4} \qquad \text{for all }\xi \geq 0.
    \end{align}
    Applying the divergence free condition of $u$, integrating by parts in the first equation of \eqref{1} and invoking Young's inequality and \eqref{L1.1}, we obtain that
    \begin{align} \label{L1.2}
        -\frac{d}{dt}\int_\Omega \log (n+B) &= -\int_\Omega \frac{1}{n+B} \left ( \Delta n -\chi \nabla \cdot (n S(n)\nabla c) - u\cdot \nabla c  \right ) \notag \\
        &= - \int_\Omega \frac{|\nabla n|^2}{(n+B)^2}   + \int_\Omega \frac{nS(n)}{(n+B)^2}\nabla n \cdot \nabla c \notag \\
        &\leq - \frac{3}{4}\int_\Omega \frac{|\nabla n|^2}{(n+B)^2} + \int_\Omega \frac{n^2S^2(n)}{(n+B)^2}|\nabla c|^2 \notag \\
        &\leq - \frac{3}{4}\int_\Omega \frac{|\nabla n|^2}{(n+B)^2} + \frac{1}{16} \int_\Omega |\nabla c|^2 \qquad \text{for all } t\in (0,T_{\rm max}). 
    \end{align}
    Multiplying the second equation of $c$ and applying Lemma \ref{Lnc} and Lemma \ref{Moser-Trudinger2}[\eqref{Moser-Trudinger2-2}] with $\epsilon=1$ and $M=1$, we can find $c_1>0$ and $c_2>0$ such that 
    \begin{align} \label{L1.3}
        \int_\Omega |\nabla c|^2 +\int_\Omega c^2 &= \int_\Omega nc \notag \\
        &\leq \frac{4\pi}{m} \int_\Omega n \log (n+B)+c_1 \notag \\
        &\leq 4 \int_\Omega  \frac{|\nabla n|^2}{(n+B)^2}+c_2.
    \end{align}
    Combining \eqref{L1.2} and \eqref{L1.3}, it follows that 
    \begin{align*}
          -\frac{d}{dt}\int_\Omega \log (n+B) + \frac{1}{2}\int_\Omega \frac{|\nabla n|^2}{(n+B)^2} \leq  \frac{c_2}{16} \qquad \text{for all } t\in (0,T_{\rm max}).
    \end{align*}
    Integrating this over $(t,t+\tau)$ and noting that $\log(n+B)\geq 0$ since $B\geq 1$ and $\log(n+B) \leq n+B$, we have
    \begin{align}\label{L1.4}
       \frac{1}{2} \int_t^{t+\tau }\int_\Omega \frac{|\nabla n|^2}{(n+B)^2} &= \frac{c_2\tau }{16}+ \int_\Omega \log(n(\cdot,t+\tau)+B) -  \int_\Omega \log(n(\cdot,t)+B) \notag \\
       &\leq \frac{c_2\tau }{16}+ \int_\Omega \left (n(\cdot,t+\tau)+B \right ) \notag \\
       &\leq c_3 \qquad \text{for all }t\in (0,T_{\rm max}-\tau),
    \end{align}
    where $c_3=\frac{c_2\tau }{16} +m+B|\Omega|$. Therefore, using \eqref{L1.4} and applying Lemma \ref{Moser-Trudinger2}[\eqref{Moser-Trudinger2-2}] with $\epsilon=M=1$, there exist $c_4>0$ and $c_5>0$ such that
    \begin{align}\label{L1.5}
      \int_t^{t+\tau}  \int_\Omega n \log (n+B) &\leq \frac{m}{\pi} \int_t^{t+\tau }\int_\Omega \frac{|\nabla n|^2}{(n+B)^2} +c_4 \notag \\
      &\leq c_5 \qquad \text{for all }t\in (0,T_{\rm max}-\tau).
    \end{align}
    From \eqref{L1.4} and \eqref{L1.5}, and noting that 
    \begin{align*}
        \int_\Omega\frac{|\nabla n|^2}{(n+1)^2} \leq B^2 \int_\Omega\frac{|\nabla n|^2}{(n+B)^2} 
    \end{align*}
    and 
    \begin{align*}
        \int_\Omega n \log(n+1) \leq \int_\Omega n \log(n+B),
    \end{align*}
    \eqref{L1-1} and \eqref{L1-2} are followed. The proof is now complete.
\end{proof}
As a consequence of the previous lemma and the Moser--Trudinger type inequality established in Lemma~\ref{Moser-Trudinger1}, we are now able to derive an $L^2$ bound for $u$. While the proof can be found in Lemmas 4.7 and 4.8 of \cite{Winkler_SIAM}, we provide a detailed proof here for completeness.

\begin{lemma} \label{L2u}
 There exist positive constant $C$ such that
    \begin{align} \label{L2u-1}
        \int_\Omega |u(\cdot,t)|^2 \leq C \qquad \text{for all } t \in (0,T_{\rm max})
    \end{align}
    and 
    \begin{align} \label{L2u-2}
          \int_t^{t+\tau } \int_\Omega |\nabla u(\cdot,s)|^2 \,ds \leq C \qquad \text{for all } t \in (0,T_{\rm max}-\tau)
    \end{align}
    as well as 
    \begin{align} \label{L2u-3}
          \int_t^{t+\tau } \int_\Omega | u(\cdot,s)|^4 \,ds \leq C \qquad \text{for all } t \in (0,T_{\rm max}-\tau),
    \end{align}
     where $\tau = \min \left \{ 1, \frac{T_{\rm max}}{2} \right \}$.
    \end{lemma}

\begin{proof}
    Thanks to Poincaré inequality, there exists $c_1>0$ such that 
    \begin{align} \label{L2u.1}
        \int_\Omega |\nabla u|^2 \geq c_1 \int_\Omega |u|^2 \qquad \text{for all }t\in (0,T_{\rm max}).
    \end{align}
  Abbreviating $c_2= \left \| \nabla \phi \right \|_{L^\infty(\Omega)} $, testing the third equation of \eqref{1} against $u=(u_1,u_2)$ and invoking Young's inequality, we obtain that 
    \begin{align}\label{L2u.2}
        \frac{d}{dt}\int_\Omega |u|^2 + \int_\Omega |\nabla u|^2 &= \int_\Omega n \nabla \phi \cdot u  \notag \\
        &\leq c_2 \int_\Omega n |u|  \notag \\
        &\leq c_2 \int_\Omega n |u_1| +  c_2 \int_\Omega n |u_2|.
    \end{align}
   By Hölder's inequality and \eqref{L2u.1}, it follows that 
    \begin{align*}
        \left \{ \int_\Omega |u| \right \}^2 \leq |\Omega| \int_\Omega |u|^2 \leq c_1|\Omega| \int_\Omega |\nabla u|^2.
    \end{align*}
    Using this and Lemma \ref{Moser-Trudinger2-1} with $\epsilon=1$ and $a>0$ yields that 
    \begin{align}\label{L2u.3}
         c_2 \int_\Omega n |u_1| +  c_2 \int_\Omega n |u_2| &\leq \frac{2c_2}{a}\int_\Omega n \ln \frac{n}{\bar{n}}+ \frac{c_2a}{2 \pi} \cdot \left \{ \int_\Omega n \right \} \cdot \int_\Omega |\nabla u|^2 \notag \\
         &\quad+2Mac_2 \cdot \left \{ \int_\Omega n \right \}   \cdot \left \{ \int_\Omega |u| \right \}^2 +\frac{2c_2M}{a} \int_\Omega n \notag \\
         &\leq  \frac{2c_2}{a}\int_\Omega n \ln \frac{n}{\bar{n}}+ \frac{c_2a}{2 \pi} \cdot \left \{ \int_\Omega n \right \} \cdot \int_\Omega |\nabla u|^2 \notag \\
         &\quad+2Mac_1c_2 |\Omega| \cdot \left \{ \int_\Omega n \right \}   \cdot  \int_\Omega |\nabla u|^2  +\frac{2c_2M}{a} \int_\Omega n.
    \end{align}
    Employing Lemma \ref{Local}[\eqref{Local.2}] and choosing $a$ such that $\frac{c_2a m}{2 \pi}+2Mac_1c_2|\Omega|m= \frac{1}{2}$ where $m =\int_\Omega n_0 $ 
    \begin{align}\label{L2u.4}
         c_2 \int_\Omega n |u_1| +  c_2 \int_\Omega n |u_2| &\leq \frac{2c_2}{a}\int_\Omega n \ln n + \frac{1}{2}\int_\Omega |\nabla u|^2 -\frac{2c_2}{a} \left \{  \int_\Omega n \right \} \ln \left ( \int_\Omega n \right ) \notag \\
         &\quad + \frac{2c_2 \ln|\Omega|}{a} \int_\Omega n + \frac{2c_2Mm}{a} \notag \\
         &\leq \frac{2c_2}{a}\int_\Omega n \ln (n+1) + \frac{1}{2}\int_\Omega |\nabla u|^2 +c_3,
    \end{align}
    where the last inequality holds due to $x\ln x\geq -\frac{1}{e}$ for all $x>0$ and \[c_3= \frac{2c_2a}{e}+ \frac{2c_2m}{a} \max \left \{ \ln |\Omega|,0 \right \}+ \frac{2c_2Mm}{a}.\] Combining \eqref{L2u.1}, \eqref{L2u.2} \eqref{L2u.4}, we arrive at
    \begin{align}\label{L2u.5}
        \frac{d}{dt}\int_\Omega |u|^2+\frac{c_1}{4}\int_\Omega |u|^2+\frac{1}{4}\int_\Omega |\nabla u|^2 \leq \frac{2c_2}{a}\int_\Omega n \ln (n+1) +c_3 \qquad \text{for all }t\in (0,T_{\rm max}).
    \end{align}
   This, together with Lemma \ref{L1}[\eqref{L1-1}] and Lemma \ref{ODI}  implies that \eqref{L2u-1}. Integrating \eqref{L2u.5} over $(t,t+\tau )$ and employing \eqref{L2u-1} and Lemma \ref{L1}[\eqref{L1-1}], there exists $c_4>0$ such that
    \begin{align*}
        \frac{1}{4}\int_t^{t+\tau }\int_\Omega |\nabla u|^2 \leq \frac{2c_2}{a} \int_t^{t+\tau }\int_\Omega n \log (n+1) +c_3 \tau +\int_\Omega |u(\cdot,t)|^2 \leq c_4\quad\text{for all }t\in (0,T_{\rm max}-\tau),
    \end{align*}
    which entails \eqref{L2u-2}.  By invoking Gagliardo–Nirenberg interpolation inequality and using \eqref{L2u-1} and \eqref{L2u-2}, one can find $c_5$ and $c_6$ 
    \begin{align*}
        \int_t^{t+\tau }\int_\Omega | u|^4 &\leq  c_5 \int_t^{t+\tau } \left (\int_\Omega |u|^2 \right ) \cdot \int_\Omega |\nabla u|^2 +c_5 \int_t ^{t+\tau} \left ( \int_\Omega |u|^2 \right )^2 \notag \\
        &\leq c_6 \qquad \text{for all }t\in (0,T_{\rm max}-\tau).
    \end{align*}
    This proves \eqref{L2u-3} and thereby completes the proof.
\end{proof}

In preparation to proof of Lemma \ref{LlolL}, we establish the following lemma allowing us to control quantities involving $\Delta c$.
\begin{lemma} \label{LM}
    There exist $C_1>0$ and $C_2>0$ such that 
    \begin{align} \label{LM-1}
        \int_\Omega |\Delta c|^2 \leq 4 \int_\Omega n^2 + C_1 \left \{ \int_\Omega |u|^4 \right \} \cdot\left \{ \int_\Omega n \log (n+1) \right \}+C_2 \int_\Omega |u|^4,
    \end{align}
    for all $t\in (0,T_{\rm max})$
\end{lemma}
\begin{proof}
    Multiplying the second equation of \eqref{1} by $\Delta c$ and applying Young's inequality yields
    \begin{align} \label{LM.2}
        \int_\Omega |\Delta c|^2 &= \int_\Omega u \cdot \nabla c \, \Delta c - \int_\Omega n \Delta c - \int_\Omega |\nabla c|^2 \notag \\
        &\leq \frac{1}{2}\int_\Omega |\Delta c|^2 + \int_\Omega |u|^2 |\nabla c|^2 + \int_\Omega n^2.
    \end{align}
     By Gagliardo–Nirenberg interpolation inequality with elliptic regularity theory, there exists $c_1>0$ such that 
    \begin{align} \label{LM.1}
        \left \| \nabla c \right \|^2_{L^4(\Omega)} \leq  c_1 \left \| \Delta c \right \|_{L^2(\Omega)} \left \| \nabla c \right \|_{L^2(\Omega)} \qquad \text{for all } t\in (0,T_{\rm max}).
    \end{align}
    In light of Hölder's inequality, Young's inequality, \eqref{LM.1} and Lemma \ref{Lnc}, there exists $c_2>0$ such that  
    \begin{align}\label{LM.3}
        \int_\Omega |u|^2 |\nabla c|^2 &\leq \left \| u \right \|_{L^4(\Omega)}^2\left \| \nabla c \right \|_{L^4(\Omega)}^2  \notag\\
        &\leq c_1  \left \| u \right \|_{L^4(\Omega)}^2\left \| \nabla c \right \|_{L^2(\Omega)} \left \| \Delta c \right \|_{L^2(\Omega)} \notag \\
        &\leq \frac{1}{4} \int_\Omega |\Delta c|^2 + c_1^2 \left \{ \int_\Omega |u|^4 \right \} \cdot \int_\Omega |\nabla c|^2 \notag \\
              &\leq \frac{1}{4} \int_\Omega |\Delta c|^2 + c_1^2 \left \{ \int_\Omega |u|^4 \right \} \cdot \int_\Omega nc \notag \\
        &\leq \frac{1}{4} \int_\Omega |\Delta c|^2 + c_1^2 \left \{ \int_\Omega |u|^4 \right \} \cdot \left \{ \frac{4\pi}{m}\int_\Omega n \log(n+1)+c_2 \right \},
    \end{align}
    where we used the the fact that $\int_\Omega |\nabla c|^2 \leq \int_\Omega nc$ in the second last inequality. Combining \eqref{LM.2} and \eqref{LM.3}, we obtain \eqref{LM-1}, which finishes the proof.
\end{proof}

To prepare for Lemma \ref{Lpn}, we derive an uniform in time $L^p$ bound for $c$. 

\begin{lemma} \label{Lpc}
    For any $p>1$, there exists $C=C(p)>0$ such that 
    \begin{align*}
        \int_\Omega c^p(\cdot,t) \leq C \qquad \text{for all }t \in (0,T_{\rm max}).
    \end{align*}
\end{lemma}

\begin{proof}
        Testing the second equation of \eqref{1} by $(c+1)^{q-1}$ with $q\in (0,1)$, we find that 
    \begin{align*}
        \left (1 -q \right ) \int_\Omega (c+1)^{q-2} |\nabla c|^2 &= \int_\Omega c(c+1)^{q-1} - \int_\Omega n (c+1)^{q-1} \notag  \\
        &\leq \int_\Omega c\qquad \text{for all }t\in (0,T_{\rm max}).
    \end{align*}
    This, together with Lemma \ref{Local}[\eqref{Local.2}] entails that 
    \begin{align*}
        \int_\Omega |\nabla (c+1)^\frac{q}{2}|^2 \leq c_1 \qquad \text{for all }t\in (0,T_{\rm max}) ,
    \end{align*}
    for some $c_1>0$. Owning to Sobolev's inequality and Lemma \ref{Local}[\eqref{Local.2}], there exist $c_2>0$ and $c_3>0$ such that
    \begin{align*}
        \int_\Omega (c+1)^p &\leq  c_2 \left (  \int_\Omega |\nabla (c+1)^\frac{q}{2}|^2 \right )^{\frac{p}{q}} + c_2\left (\int_\Omega (c+1) \right )^p \notag \\
        &\leq c_3  \qquad \text{for all }t\in (0,T_{\rm max}),
    \end{align*}
    which completes the proof.
\end{proof}
\subsection{$L^p$ boundedness for $n$} \label{S3-2}
We are now ready to present the second main ingredient of our analysis for the proof of Theorem \ref{thm1}.
\begin{lemma} \label{LlolL}
    There exists $C>0$ such that 
    \begin{align} \label{LlolL-1}
        \int_\Omega n \log (n+1) \leq C \qquad \text{for all }t\in (0,T_{\rm max})
    \end{align}
    and 
    \begin{align} \label{LlolL-2}
        \int_t ^{t+\tau } \int_\Omega n^2 \leq C\qquad \text{for all }t\in (0,T_{\rm max}-\tau ),
    \end{align}
    where $\tau = \min \left \{1, \frac{T_{\rm max}}{2} \right \}$.
\end{lemma}
\begin{proof}
    Testing the first equation of \eqref{1} by $\log n+1$ and noting that $\nabla \cdot u=0$, we obtain
    \begin{align} \label{LlolL.1}
        \frac{d}{dt}\int_\Omega n \log n &= \int_\Omega (\log n+1)(\Delta n - \nabla \cdot (nS(n)\nabla c) -u \cdot \nabla n ) \notag \\
        &= - \int_\Omega \frac{|\nabla n|^2}{n} +\int_\Omega S(n)\nabla n \cdot \nabla c \notag \\
        &=  - \int_\Omega \frac{|\nabla n|^2}{n} - \int_\Omega \psi(n)\Delta c,
    \end{align}
    where $\psi(n) = \int_0^n S(\xi )\, d\xi$. Invoking Young's inequality with $\epsilon>0$ and using  Lemma \ref{LM}[\eqref{LM-1}], there exist $c_1>0$ and $c_2>0$ such that 
    \begin{align} \label{LlolL.2}
        - \int_\Omega \psi(n)\Delta c &\leq \epsilon \int_\Omega |\Delta c|^2 + \frac{1}{4\epsilon} \int_\Omega \psi^2(n) \notag \\
        &\leq 4\epsilon \int_\Omega n^2 +c_1 \epsilon \left \{ \int_\Omega |u|^4 \right \} \cdot \int_\Omega n \log n +c_2 \epsilon \int_\Omega |u|^4+ \frac{1}{4\epsilon} \int_\Omega \psi^2(n).
    \end{align}
    Next, we show that for any $\delta>0$, there exists $C=C(\delta)>0$ such that
    \begin{align}  \label{LlolL.3}
        |\psi(\xi)| \leq \delta \xi +C \qquad \text{for all }\xi \geq 0.
    \end{align}
    Indeed, the condition \eqref{S} asserts that there exists $K>0$ such that 
    \begin{align*}
        |S(\sigma )| \leq \delta \qquad \text{for all }\sigma \geq K.
    \end{align*}
    This leads to 
    \begin{align*}
        \int_0 ^\xi |S(\sigma)| \, d\sigma \leq \int_0 ^K |S(\sigma)| \, d\sigma \qquad \text{for all }\xi \leq K
    \end{align*}
    and 
    \begin{align*}
         \int_0 ^\xi |S(\sigma)| \, d\sigma \leq  \int_0 ^K |S(\sigma)| \, d\sigma +\int_K ^\xi |S(\sigma)| \, d\sigma \qquad \text{for all }\xi > K.
    \end{align*}
    The above estimates imply that 
    \begin{align*}
         \int_0 ^\xi |S(\sigma)| \, d\sigma \leq K  \left \|S \right \|_{L^\infty((0, \infty))} + \delta \xi \qquad \text{for all }\xi \geq 0,
    \end{align*}
    which implies \eqref{LlolL.3}. Now, employing \eqref{LlolL.3} with $\delta=2\epsilon$, there exists $c_3=c_3(\epsilon)>0$ such that 
    \begin{align}\label{LlolL.4}
        \frac{1}{4\epsilon}\int_\Omega \psi^2(n) \leq \epsilon \int_\Omega n^2+c_3.
    \end{align}
    By Gagliardo–Nirenberg interpolation inequality, there exists $c_4>0$ such that 
    \begin{align} \label{LlolL.5}
        \int_\Omega n^2 &\leq c_4 \left ( \int_\Omega n \right )\cdot \int_\Omega \frac{|\nabla n|^2}{n}+c_4\left ( \int_\Omega n \right )^2 \notag \\
        &\leq c_4m \int_\Omega \frac{|\nabla n|^2}{n} +c_4m^2,
    \end{align}
    where $m =\int_\Omega n_0$. Setting 
    \begin{align*}
        y(t)= \int_\Omega n \log n + \frac{|\Omega|}{e} \qquad \text{for all }t\in (0,T_{\rm max}),
    \end{align*}
     we see that $y $ is nonnegative due to $x\log x \geq -\frac{1}{e}$ for all $x>0$.
     Choosing $\epsilon= \frac{1}{10 c_4m}$ and collecting  \eqref{LlolL.1}, \eqref{LlolL.2}, \eqref{LlolL.4} and \eqref{LlolL.5}, one can find $c_5>0$ such that 
     \begin{align*}
         y'(t) &\leq  - \int_\Omega \frac{|\nabla n|^2}{m}+ 5\epsilon \int_\Omega n^2 + c_1 \epsilon \left \{ \int_\Omega |u|^4 \right \} \cdot \int_\Omega n \log n +c_2 \epsilon \int_\Omega |u|^4+c_3 \notag \\
         &\leq \left ( 5c_4 m \epsilon-1  \right )\int_\Omega \frac{|\nabla n|^2}{m} +c_5 \left \{ \int_\Omega |u|^4 \right \} \cdot y(t)+ c_5 \int_\Omega |u|^4+c_5 \notag \\
         &\leq -\frac{1}{2}\int_\Omega \frac{|\nabla n|^2}{n} +c_5 \left \{ \int_\Omega |u|^4 \right \} \cdot y(t)+ c_5 \int_\Omega |u|^4+c_5 \qquad\text{for all }t\in (0,T_{\rm max}).
     \end{align*}
     This entails that 
     \begin{align} \label{LlolL.6}
         y'(t)+\frac{1}{2}\int_\Omega \frac{|\nabla n|^2}{n} \leq h(t)y(t)+g(t)\qquad\text{for all }t\in (0,T_{\rm max}),
     \end{align}
     where $h(t)= c_5 \int_\Omega |u|^4$ and $g(t)=c_5 \int_\Omega |u|^4+c_5$. By invoking Lemma \ref{L2u}[\eqref{L2u-3}] and Lemma \ref{L1}[\eqref{L1-1}], there exists $c_6>0$ such that 
     \begin{align*}
         \int_t^{t+\tau} h(s) \, ds \leq c_6 \qquad \text{for all }t\in (0,T_{\rm max}-\tau)
     \end{align*}
     and 
     \begin{align*}
         \int_t^{t+\tau} g(s) \, ds \leq c_6 \qquad \text{for all }t\in (0,T_{\rm max}-\tau)
     \end{align*}
     as well as
     \begin{align*}
         \int_t^{t+\tau} y(s) \, ds \leq c_6 \qquad \text{for all }t\in (0,T_{\rm max}-\tau),
     \end{align*}
      which enables us to apply Lemma \ref{ODI'} to \eqref{LlolL.6}  to deduce 
      that 
      \begin{align*}
          \int_\Omega n(\cdot,t) \log n(\cdot,t) \leq c_7 \qquad \text{for all } t\in (0,T_{\rm max}),
      \end{align*}
      for some $c_7>0$. This implies \eqref{LlolL-1} since 
      \begin{align*}
          \int_\Omega n \log (n+1) &= \int_{ \left \{ n\leq 1 \right \} }n \log (n+1) +\int_{ \left \{ n > 1 \right \} }n \log (n+1) \notag \\
          &\leq |\Omega| \log 2  + \int_{ \left \{ n > 1 \right \} }n \log (2n) \notag \\
          &\leq  |\Omega| \log 2 +m \log 2 + \int_\Omega n \log n -  \int_{ \left \{ n\leq 1 \right \} }n \log n \notag \\
          &\leq |\Omega| \log 2 +m \log 2 +c_7+ \frac{|\Omega|}{e},
      \end{align*}
      where we used $x\log x\geq -\frac{1}{e}$ in the last inequality. Integrating \eqref{LlolL.6} over $(t,t+\tau )$, we obtain that 
      \begin{align}\label{LlolL.7}
          \int_t^{t+\tau }\int_\Omega \frac{|\nabla n|^2}{n} &\leq2 \int_t^{t+\tau }h(s)y(s)\, ds+2\int_t^{t+\tau }g(s)\, ds +2y(t) \notag \\
          &\leq c_8 \qquad \text{for all }t \in (0,T_{\rm max}-\tau),
      \end{align}
      where $c_8= 2\left (c_7+\frac{|\Omega|}{e} \right )c_6+2c_6+2c_7+\frac{2|\Omega|}{e} $. From \eqref{LlolL.5} and \eqref{LlolL.7}, it follows that 
      \begin{align*}
            \int_t^{t+\tau }\int_\Omega  n^2 \leq c_4c_8m+c_4m^2 \tau,
      \end{align*}
      which proves \eqref{LlolL-2}. The proof is now complete.

\end{proof}

Next, let us derive a uniform in time $W^{1,2}$ bound for $u$ in the following lemma. Although the proof of the lemma is similar to  \cite{YM-2016}[Lemma 3.6], we still give a detailed proof here for completeness.
\begin{lemma} \label{Gradu}
    There exists $C_1>0$ such that 
    \begin{align} \label{Gradu-1}
        \int_\Omega |\nabla u (\cdot,t)|^2 \leq C_1 \qquad \text{for all }t\in (0,T_{\rm max}).
    \end{align}
    Moreover, for any $p\geq 1$ there exists $C_2=C_2(p)>0$ such that 
    \begin{align} \label{Gradu-2}
        \int_\Omega |u(\cdot,t)|^p \leq C_2\qquad \text{for all }t\in (0,T_{\rm max}).
    \end{align}
\end{lemma}

\begin{proof}
  Applying the Helmholtz projector $\mathcal{P}$ to the third equation in \eqref{1} and then taking the inner product of the resulting expression with $Au$, we obtain an identity from which, by invoking Young's inequality together with the orthogonality property of $\mathcal{P}$, it follows that 
    \begin{align} \label{Lpu.2}
        \frac{1}{2} \frac{d}{dt}\int_\Omega |A^\frac{1}{2}u|^2 &= -\int_\Omega |Au|^2 - \int_\Omega  \mathcal{P}((u\cdot \nabla)u) \cdot Au + \int_\Omega  \mathcal{P}(n \nabla \phi) \cdot Au  \notag \\
        &\leq -\frac{3}{4}\int_\Omega |Au|^2+ \int_\Omega |(u \cdot \nabla )u|^2 + \left \| \nabla \phi \right \|_{L^\infty(\Omega)} \int_\Omega n^2. 
    \end{align}
     Using Lemma \ref{L2u}[\eqref{L2u-1}], Gagliardo–Nirenberg interpolation inequality and Young's inequality, there exists $c_1>0$ and $c_2>0$ such that
    \begin{align} \label{Lpu.3}
         \int_\Omega |(u \cdot \nabla )u|^2 &\leq \left \| u \right \|^2_{L^\infty(\Omega)} \int_\Omega |\nabla u|^2 \notag \\
         &\leq c_1 \left ( \int_\Omega |u|^2 \right )^{\frac{1}{2}} \left ( \int_\Omega |Au|^2 \right )^\frac{1}{2}\int_\Omega |\nabla u|^2  \notag \\
         &\leq \frac{1}{4} \int_\Omega |Au |^2 + c_2 \left ( \int_\Omega |\nabla u|^2 \right )^2 \qquad \text{for all }t \in (0,T_{\rm max}). 
    \end{align}
    Setting 
    \begin{align*}
        h_1(t) = c_2\int_\Omega |\nabla u(\cdot,t)|^2,
    \end{align*}
    and
    \begin{align*}
        h_2(t)= \left \| \nabla \phi \right \|_{L^\infty(\Omega)} \int_\Omega n^2(\cdot,t). 
    \end{align*}
    then from Lemma \ref{L2u}  [\eqref{L2u-2}] and  Lemma \ref{LlolL}[\eqref{LlolL-2}], there exists $c_3>0$ such that
    \begin{align}\label{Lpu.4}
        \int_t ^{t+1} h_1(s)\, ds  \leq c_3 \qquad \text{for all }t \in (0,T_{\rm max})
    \end{align}
    and 
    \begin{align}\label{Lpu.5}
        \int_t ^{t+1} h_2(s)\, ds \leq c_3 \qquad \text{for all }t \in (0,T_{\rm max}).
    \end{align} 
    Collecting \eqref{Lpu.2} and \eqref{Lpu.3}, and noting that $\int_\Omega |A^\frac{1}{2}u | =\int_\Omega |\nabla u|^2$ gives us that 
    \begin{align*}
        \frac{d}{dt}\int_\Omega |\nabla u(\cdot,t)|^2 \leq h_1(t)\int_\Omega |\nabla u(\cdot,t)|^2 +h_2(t) \qquad \text{for all }t \in (0,T_{\rm max}).
    \end{align*}
    Using the estimates \eqref{Lpu.4} and \eqref{Lpu.5} and applying Lemma \ref{ODI'}, we deduce \eqref{Gradu-1}. By applying Hölder's inequality and Lemma \ref{L2u}[\eqref{L2u-1}], \eqref{Gradu-2} is followed for any $p \in (1,2)$. In case $p\geq 2$, we apply Sobolev's inequality together with Lemma \ref{L2u}[\eqref{L2u-1}] and \eqref{Gradu-1} to deduce that 
    \begin{align} \label{Lpu.1}
        \int_\Omega |u|^p \leq  c_4 \left ( \int_\Omega |\nabla u|^2 \right )^{\frac{p}{2}} +c_4 \left ( \int_\Omega |u|^2 \right )^{\frac{p}{2}} \leq c_5 \qquad \text{for all }t\in (0,T_{\rm max}), 
    \end{align}
    where $c_4>0$ and $c_5>0$.
\end{proof}

We are now able to present an $L^p$ boundedness result for $n$ as follows.
\begin{lemma} \label{Lpn}
    For any $p>1$, there exists $C=C(p)>0$ such that 
    \begin{align*}
        \int_\Omega n^p(\cdot,t) \leq C \qquad \text{for all }t\in (0,T_{\rm max}).
    \end{align*}
\end{lemma}

\begin{proof}
    Multiplying the first equation of \eqref{1}, noting that $\nabla \cdot u =0$, and using the second equation of \eqref{1}, we obtain 
    \begin{align} \label{Lpn.1}
        \frac{1}{p}\frac{d}{dt} \int_\Omega n^p &= \int_\Omega n^{p-1} \left ( \Delta n - \nabla \cdot (n S(n)\nabla c) - u \cdot \nabla n \right ) \notag \\
        &= -(p-1) \int_\Omega n^{p-2}|\nabla n|^2 + (p-1) \int_\Omega n^{p-1} S(n ) \nabla n \cdot \nabla c \notag \\
        &=  -(p-1) \int_\Omega n^{p-2}|\nabla n|^2 +  \int_\Omega  \nabla \psi (n) \cdot \nabla c \notag \\
         &=  - (p-1)\int_\Omega n^{p-2}|\nabla n|^2 -  \int_\Omega  \psi (n) \Delta c \notag \\
         &= - (p-1)\int_\Omega n^{p-2}|\nabla n|^2 +  \int_\Omega  n \psi (n) - \int_\Omega u \cdot \nabla c \, \psi(n),
    \end{align}
    where 
    \begin{align*}
        \psi(\xi ) = (p-1)\int_0 ^\xi \sigma^{p-1}S(\sigma)\, d\sigma \qquad \text{for all }\xi \geq 0.
    \end{align*}
    Integrating by parts, using the divergence free condition of $u$, and invoking Young's inequality  yields
    \begin{align} \label{Lpn.2}
        - \int_\Omega u \cdot \nabla c \, \psi(n) &= (p-1) \int_\Omega n^{p-1}S(n) c u \cdot \nabla n \notag\\
        &\leq \frac{p-1}{2} \int_\Omega n^{p-2}|\nabla n|^2 + \frac{p-1}{2} \int_\Omega n^p S^2(n) c^2 |u|^2.
    \end{align}
   By Gagliardo–Nirenberg interpolation inequality and Lemma \ref{Local}[\eqref{Local.2}], there exist $c_1>0$ and $c_2>0$ such that 
   \begin{align} \label{Lpn.3}
       \int_\Omega n^{p+1} &\leq c_1 \left ( \int_\Omega n \right ) \cdot \int_\Omega n^{p-2}|\nabla n|^2 +c_1 \left ( \int_\Omega n \right ) ^{p+1} \notag \\
       &\leq c_2 \int_\Omega n^{p-2}|\nabla n|^2 +c_2.
   \end{align}
   Since $S$ is bounded due to \eqref{S}, we apply Young's inequality together with \eqref{Local.2} and Lemma \ref{Gradu}[\ref{Gradu-2}] and Lemma \ref{Lpc} to deduce that there exist positive constants $c_3, c_4,$ and $c_5$ fulfilling 
   \begin{align} \label{Lpn.4}
       \frac{p-1}{2} \int_\Omega n^p S^2(n) c^2 |u|^2 &\leq c_3 \int_\Omega n^p c^2 |u|^2 \notag \\
       &\leq \frac{p-1}{4c_2}\int_\Omega n^{p+1}+c_4 \int_\Omega c^{2p+2}|u|^{2p+2} \notag \\
       &\leq \frac{p-1}{4}\int_\Omega n^{p-2}|\nabla n|^2 + \frac{p-1}{4} +c_4 \int_\Omega c^{4p+4} +c_4 \int_\Omega |u|^{4p+4} \notag \\
       &\leq \frac{p-1}{4}\int_\Omega n^{p-2}|\nabla n|^2 +c_5.
   \end{align}
   Next, we show that for any $\delta>0$, there exists $C=C(\delta)>0$ such that
   \begin{align} \label{Lpn.5}
       |\psi(\xi)| \leq \delta \xi + C \qquad \text{for all }\xi \geq 0.
   \end{align}
   Indeed, the condition \eqref{S} implies that $S$ is bounded and satisfies that for any $\delta>0$, there exists $K=K(\delta)>0$ such that
   \begin{align*}
       |S(\xi)| \leq \delta \qquad \text{for all }\xi \geq K.
   \end{align*}
   This further leads to 
    \begin{align*}
        (p-1)\int_0^\xi \sigma^{p-1}|S(\sigma)|\, d\sigma \leq \frac{p-1}{p} \left \| S \right \|_{L^\infty((0,\infty))} K^p\qquad \text{for all }\xi < K.
    \end{align*}
    and
     \begin{align*}
        (p-1)\int_0^\xi \sigma^{p-1}|S(\sigma)|\, d\sigma &\leq (p-1) \int_0^K \sigma^{p-1}|S(\sigma)|\, d\sigma +(p-1)\int_K ^\xi \sigma^{p-1}|S(\sigma)|\, d\sigma \notag \\
        &\leq \frac{p-1}{p} \left \| S \right \|_{L^\infty((0,\infty))} K^p +\frac{p-1}{p} \delta \xi^p  \qquad \text{for all }\xi \geq K.
    \end{align*}
    Therefore, we have
    \begin{align*}
         (p-1)\int_0^\xi \sigma^{p-1}|S(\sigma)|\, d\sigma \leq \delta \xi^p +\left \| S \right \|_{L^\infty((0,\infty))} K^p  \qquad \text{for all }\xi \geq 0,
    \end{align*}
    which implies \eqref{Lpn.5}. Combining \eqref{Lpn.5}, \eqref{Lpn.3} and \eqref{Local.2} and employing Young's inequality, one can find $c_6>0$ and $c_7>0$ such that 
    \begin{align} \label{Lpn.6}
       \frac{1}{p}\int_\Omega n^p+ \int_\Omega n \psi(n) &\leq \frac{p-1}{4c_2} \int_\Omega n^{p+1}+c_6 \int_\Omega n \notag \\
        &\leq \frac{p-1}{4} \int_\Omega n^{p-2}|\nabla n|^2 + c_7.
    \end{align}
    Collecting \eqref{Lpn.1}, \eqref{Lpn.2}, \eqref{Lpn.4}, and \eqref{Lpn.6}, we arrive at
    \begin{align*}
        \frac{1}{p} \frac{d}{dt} \int_\Omega n^p+\frac{1}{p}\int_\Omega n^p \leq c_8 \qquad \text{for all }t\in (0,T_{\rm max}),
    \end{align*}
    where $c_8=c_5+c_7$. This, together with Gronwall's inequality implies that 
    \begin{align*}
        \int_\Omega n^p(\cdot,t) \leq \max \left \{ \int_\Omega n_0^p, pc_8 \right \}\qquad \text{for all }t\in (0,T_{\rm max}),
    \end{align*}
    which completes the proof.
\end{proof}

\subsection{Proof of the main result for $N=2$}  \label{S3-3}

As a consequence of $L^p$ boundedness of $n$ established in Lemma \ref{Lpn}, we can now derive an $L^\infty$ bound for $u$ as the following lemma.

\begin{lemma}\label{uLinf}
    For any $\beta \in \left ( \frac{1}{2}, 1 \right )$, there exists $C>0$ such that 
    \begin{align}\label{uLinf-1}
        \left \| A^\beta u (\cdot,t) \right \|_{L^2(\Omega)} \leq C \quad \text{for all }t \in (0,T_{\rm max})
    \end{align}
    and 
    \begin{align}\label{uLinf-2}
         \left \|u (\cdot,t) \right \|_{L^\infty(\Omega)} \leq C \quad \text{for all }t \in (0,T_{\rm max}).
    \end{align}
\end{lemma}    
\begin{proof}
    Noting that $\beta <1$, we can choose $r\in (1,2)$ such that $\frac{1}{r}< \frac{3}{2}-\beta$, and apply known smoothing properties of Stokes semigroup $(e^{-tA})_{t\geq 0}$ (see \cite{Giga1986}[p.201] and \cite{Winkler+2015}[Lemma 3.1] ) and continuity features of $\mathcal{P}$ to find $c_1>0$ and $\lambda>0$ such that 
   \begin{align} \label{uLinf.1}
\left\| A^\beta u(\cdot,t) \right\|_{L^2(\Omega)}
&= \left\| A^\beta e^{-tA}u_0
+ \int_0^t A^\beta e^{-(t-s)A} \mathcal{P}\big[n(\cdot,s)\nabla \phi\big] \, ds \right. \notag \\
&\qquad\left.
- \int_0^t A^\beta e^{-(t-s)A} \mathcal{P}\big[(u(\cdot,s)\cdot\nabla)u(\cdot,s)\big] \, ds
\right\|_{L^2(\Omega)} \notag \\
&\leq \left \| A^\beta u_0 \right \|_{L^2(\Omega)} +c_1\int_0^t (t-s)^{-\beta } e^{-\lambda(t-s)}\left \| n(\cdot,s) \right \|_{L^2(\Omega)} \notag \\
&\quad+c_1\int_0^t (t-s)^{-\beta - \frac{1}{r}+\frac{1}{2}} e^{-\lambda(t-s)} \left \| (u(\cdot,s)\cdot\nabla)u(\cdot,s) \right \|_{L^r(\Omega)} \qquad \text{for all }t\in (0,T_{\rm max}).
\end{align}
Applying Lemma \ref{Gradu} and Hölder's inequality, there exists $c_2>0$ such that
\begin{align} \label{uLinf.2}
    \left \| (u\cdot\nabla)u \right \|_{L^r(\Omega)} \leq \left \| u \right \|_{L^{\frac{2r }{2-r}}(\Omega)} \left \| \nabla u \right \|_{L^2(\Omega)} \leq c_2 \qquad \text{for all }t\in (0,T_{\rm max}).
\end{align}
Combining \eqref{uLinf.1} and \eqref{uLinf.2}, and apply Lemma \ref{Lpn}, there exists $c_3>0$ such that
\begin{align}\label{uLinf.3}
    \left\| A^\beta u(\cdot,t) \right\|_{L^2(\Omega)} &\leq \left \| A^\beta u_0 \right \|_{L^2(\Omega)} +c_1 \sup_{t\in (0,T_{\rm max})} \left \| n(\cdot,t) \right \|_{L^2(\Omega )} \int_0^\infty z^{-\beta}e^{-\lambda z}\, dz \notag \\
    &\quad+c_1c_2 \int_0 ^\infty z^{-\beta -\frac{1}{r}+\frac{1}{2}}e^{-\lambda z}\, dz \notag \\
    &\leq  c_3 \qquad \text{for all }t\in (0,T_{\rm max}). 
\end{align}
As a consequence of the continuity embedding $D(A^\beta) \to L^\infty (\Omega; \mathbb{R}^2)$ and \eqref{uLinf.3}, there exists $c_4>0$ satisfying 
\begin{align*}
     \left\|  u(\cdot,t) \right\|_{L^\infty(\Omega)} \leq c_4  \left\| A^\beta u(\cdot,t) \right\|_{L^2(\Omega)}\leq c_4c_3\qquad \text{for all }t\in (0,T_{\rm max}). 
\end{align*}
The proof is now complete.

\end{proof}

Thanks to the $L^p$ boundedness of $n$, the $L^\infty$ boundedness of $u$, and the regularity theory for elliptic equations, we can now establish the uniform boundedness of $c$ in $W^{1,\infty}(\Omega)$.

\begin{lemma} \label{cbdd'}
    There exists $C>0$ such that 
    \begin{align*}
        \left \| c(\cdot,t) \right \|_{W^{1, \infty}(\Omega)} \leq C \qquad \text{for all }t\in (0,T_{\rm max}).
    \end{align*}
\end{lemma}
\begin{proof}
     Testing the second equation of \eqref{1} by $c$, using Lemma \ref{Lpn} and applying Young's inequality yields 
    \begin{align*}
        \int_\Omega |\nabla c|^2 +\int_\Omega c^2 &= \int_\Omega nc \notag \\
        &\leq \frac{1}{2}\int_\Omega n^2 + \frac{1}{2} \int_\Omega c^2,
    \end{align*}
    which leads to
    \begin{align} \label{cbdd'.1}
        \left \| c(\cdot,t) \right \|_{W^{1,2}(\Omega)} \leq c_1 \qquad \text{for all } t\in (0,T_{\rm max}),
    \end{align}
    for some $c_1>0$. Applying Lemma \ref{Lpn} with $p=3$ and the uniform boundedness of $u$ in Lemma \ref{uLinf}[\eqref{uLinf-2}], one can find $c_2>0$ and $c_3>0$ such that 
    \begin{align*}
         \left \| n(\cdot,t) \right \|_{L^3(\Omega)} \leq c_2 \qquad \text{for all }t\in (0,T_{\rm max}),
    \end{align*}
    and 
    \begin{align*}
         \left \| u(\cdot,t) \right \|_{L^\infty(\Omega)} \leq c_3 \qquad \text{for all }t\in (0,T_{\rm max}).
    \end{align*}
    Owning to $L^p$ elliptic regularity theory, \eqref{cbdd'.1}, and Gagliardo–Nirenberg interpolation inequality, we can find positive constants  $c_4, c_5$ and $c_6$ such that 
    \begin{align*} 
        \left \| c \right \|_{W^{2,3}(\Omega)} &\leq c_4 \left \| n \right \|_{L^3(\Omega)} +c_4 \left \| u \cdot \nabla c  \right \|_{L^3(\Omega)}  \notag \\
        &\leq c_2c_4 +c_3c_4  \left \| \nabla c  \right \|_{L^3(\Omega)} \notag \\
        &\leq c_2c_4 +c_3c_4c_5  \left \| c \right \|_{W^{2,3}(\Omega)}^{\frac{1}{2}} \left \| \nabla c  \right \|^{\frac{1}{2}}_{L^2(\Omega)} \notag \\
        &\leq \frac{1}{2} \left \| c \right \|_{W^{2,3}(\Omega)}  +c_6 \qquad\text{for all } t\in (0,T_{\rm max}),
    \end{align*}
    which entails that 
    \begin{align*} 
        \left \| c(\cdot,t) \right \|_{W^{2,3}(\Omega)} \leq 2c_6 \qquad \text{for all } t\in (0,T_{\rm max}).
    \end{align*}
    This, together with the Sobolev embedding $W^{2,3}(\Omega) \to W^{1, \infty}(\Omega)$ yields 
     \begin{align*} 
        \left \| c(\cdot,t) \right \|_{W^{1,\infty}(\Omega)} \leq c_7 \left \| c(\cdot,t) \right \|_{W^{2,3}(\Omega)} \leq 2c_7c_6 \qquad \text{for all } t\in (0,T_{\rm max}).
    \end{align*}
    for some $c_7>0$. The proof is now complete.
\end{proof}

With the uniform boundedness of $\nabla c$ at hand, and by employing a Moser-type iterative argument, we can now establish an $L^\infty$ bound for $n$ as follows.

\begin{lemma}\label{nbdd'}
    There exists $C>0$ such that 
    \begin{align*}
        \left \| n(\cdot,t) \right \|_{L^{ \infty}(\Omega)} \leq C \qquad \text{for all }t\in (0,T_{\rm max}).
    \end{align*}
\end{lemma}
\begin{proof}
      Testing the first equation of \eqref{1} by $pn^{p-1}$, applying Lemma \ref{cbdd'}, Young's inequality and noting that $S$ is bounded due to \eqref{S}, there exists $c_1>0$ independent of $p$ satisfying
    \begin{align*}
        \frac{d}{dt}\int_\Omega n^p &=-p(p-1) \int_\Omega n^{p-2}|\nabla n|^2 +p(p-1) \int_\Omega n^{p-1}S(n) \nabla n \cdot \nabla c \notag \\
        &\leq -\frac{p(p-1)}{2}) \int_\Omega n^{p-2}|\nabla n|^2 + \frac{p(p-1)}{2}\int_\Omega n^p S^2(n)|\nabla c|^2 \notag \\
        &\leq -\frac{p(p-1)}{2}) \int_\Omega n^{p-2}|\nabla n|^2  + c_1 p(p-1) \int_\Omega n^p \qquad \text{for all } t\in (0,T_{\rm max}). 
    \end{align*}
    Owning to Lemma \ref{Lpn}, we have  that $n \in L^\infty \left ((0,T_{\rm max});L^p(\Omega) \right )$ for any $p>1$. Using this and standard Moser-type iterative argument (see  \cite{Alikakos1}, \cite{Alikakos2}) or \cite{Winkler-2011}[Lemma A.1]), the lemma is thus followed. The proof is now complete.
\end{proof}

Having prepared the above Lemmas, we are now ready to present the proof of our first main result.
\begin{proof}[Proof of Theorem \ref{thm1}]
 The global solvability and boundedness feature of the solution $(n,c,u)$ \eqref{thm1-1} is an immediate consequence of Lemmas \ref{Local}[\eqref{Local.1}], \ref{nbdd'}, \ref{uLinf} and \ref{cbdd'}.
\end{proof}

\section{The Three Dimension Keller-Segel-Stokes Systems}\label{S4}
In this section, we consider the case $N = 3$ and assume that $S$ satisfies condition~\eqref{S} with $\alpha > \frac{1}{3}$. We begin by establishing some a priori estimates in the next subsection. Subsequently, we derive the cornerstone of our analysis, namely the $L^p$ estimate for $n$, and conclude with the proof of Theorem~\ref{thm2} in Subsection~\ref{S4-2}.

\subsection{A Priori Estimates} \label{S4-1}
By combining the Stokes semigroup estimates with the time-uniform boundedness of $n \nabla \phi$, we obtain the following $L^{p}$-estimate for $u$. As the argument is standard and can be found, for example, in \cite{YZ2016}, Lemma~2.5, and \cite{Winkler-2018}, Lemma~2.3, we omit the detailed proof here to avoid redundancy.
 
\begin{lemma} \label{L2}
    For any $p \in (1,3)$, there exists $C>0$ such that 
    \begin{align}
        \left \|u(\cdot,t) \right \|_{L^p(\Omega)} \leq C\qquad \text{for all }t\in (0,T_{\rm max}).
    \end{align}
\end{lemma}
In preparation for the proof of Lemma~\ref{L5}, we first establish the following lemma, which provides an $L^p$ bound for $c$ with $p < 3$.

\begin{lemma} \label{L3}
        For any $p \in (1,3)$, there exists $C>0$ such that 
    \begin{align}
        \left \|c(\cdot,t) \right \|_{L^p(\Omega)} \leq C\qquad \text{for all }t\in (0,T_{\rm max}).
    \end{align}
\end{lemma}
\begin{proof}
    Multiplying the second equation of \eqref{1} by $(c+1)^{\frac{p}{3}-1}$ and integrating by parts yields
    \begin{align*}
        \left (1 -\frac{p}{3} \right ) \int_\Omega (c+1)^{\frac{p}{3}-2} |\nabla c|^2 &= \int_\Omega c(c+1)^{\frac{p}{3}-1} - \int_\Omega n (c+1)^{\frac{p}{3}-1} \notag  \\
        &\leq \int_\Omega c\qquad \text{for all }t\in (0,T_{\rm max}).
    \end{align*}
    This, together with Lemma \ref{Local}[\eqref{Local.2}] implies that 
    \begin{align*}
        \int_\Omega |\nabla (c+1)^\frac{p}{6}|^2 \leq c_1 \qquad \text{for all }t\in (0,T_{\rm max}) ,
    \end{align*}
    for some $c_1>0$. Applying Sobolev's inequality and Lemma \ref{Local}[\eqref{Local.2}], we can find $c_2>0$ and $c_3>0$ such that
    \begin{align*}
        \int_\Omega (c+1)^p &\leq  c_2 \left (  \int_\Omega |\nabla (c+1)^\frac{p}{6}|^2 \right )^{3} + c_2\left (\int_\Omega (c+1) \right )^p \notag \\
        &\leq c_3  \qquad \text{for all }t\in (0,T_{\rm max}),
    \end{align*}
    which completes the proof.
\end{proof}

The following lemma provides us a vital estimate allowing us to approximate $c$ as $n^{\frac{1}{\lambda}}$ with $\lambda<3$.  

\begin{lemma} \label{L4}
    For $\lambda \in [1,3)$ and $p> \lambda$, there exists $C=C(p,\lambda)>0$ such that 
    \begin{align}
        \int_\Omega c^{ p} \leq C \int_\Omega n^{\frac{p}{\lambda}}+C\qquad \text{for all }t\in (0,T_{\rm max}).
    \end{align}
\end{lemma}
\begin{proof}
    Using Lemma \ref{L3} and noting that $\frac{3(\lambda-1)}{2}\in [0,3)$ when $1 \leq \lambda<3$, we can find $c_1>0$ such that 
    \begin{align*}
        \int_\Omega c^{\frac{3(\lambda-1)}{2}}(\cdot,t) \leq c_1 \qquad \text{for all }t\in (0,T_{\rm max}).
    \end{align*}
    By combining Gagliardo–Nirenberg interpolation inequality the above estimate, there exist $c_2>0$ and $c_3>0$ such that 
    \begin{align} \label{L4.1}
        \int_\Omega c^p&\leq c_2 \int_\Omega c^{p-\lambda-1}|\nabla c|^2 \left ( \int_\Omega c^{\frac{3(\lambda-1)}{2}}  \right )^{\frac{2p-3\lambda+3}{3p-3\lambda+\frac{3}{2}}} +c_2 \left ( \int_\Omega c^{\frac{3(\lambda-1)}{2}} \right )^{\frac{2p}{3(\lambda-1)}} \notag \\
        &\leq c_3 \int_\Omega c^{p-\lambda-1}|\nabla c|^2+c_3.
    \end{align}
    Multiplying the second equation of \eqref{1} by $c^{p-\lambda}$ and integrating by parts yields
    \begin{align*}
        (p-\lambda) \int_\Omega c^{p-\lambda-1}|\nabla c|^2 + \int_\Omega c^{p-\lambda+1}= \int_\Omega n c^{p-\lambda}.
    \end{align*}
    Dropping the second term, utilizing \eqref{L4.1} and applying Young's inequality, we can find $c_4>0$ such that
    \begin{align*}
        \frac{p-\lambda}{c_3}\int_\Omega c^p -p+\lambda &\leq \int_\Omega nc^{p-\lambda} \notag \\
        &\leq \frac{p-\lambda}{2c_3}\int_\Omega c^p + c_4 \int_\Omega n^{\frac{p}{\lambda}}.
    \end{align*}
    This further entails that 
    \begin{align*}
        \int_\Omega c^p \leq \frac{2c_3c_4}{p-\lambda}\int_\Omega n^{\frac{p}{\lambda}}+ 2c_3,
    \end{align*}
    which completes the proof.
\end{proof}

Building on the preliminary insight provided by the previous lemma, we can now derive the central result of our analysis.

\begin{lemma} \label{L5}
    For any $p>1$, there exists $C>0$ such that 
    \begin{align}
        \int_\Omega n^p(\cdot,t) \leq C \qquad \text{for all }t\in (0,T_{\rm max}).
    \end{align}
\end{lemma}

\begin{proof}
   As $\Omega$ is bounded, without loss of generality, we assume further that $p>\max \left \{2\alpha,1 \right \}$. Testing the first equation of \eqref{1} by $n^{p-1}$, together with several applications of integration by parts and using the second equation of \eqref{1}, we obtain 
    \begin{align} \label{L5.1}
        \frac{d}{dt} \int_\Omega n^p +p(p-1)\int_\Omega n^{p-2}|\nabla n|^2&= p(p-1) \int_\Omega n^{p-1}S(n) \nabla n \cdot \nabla c \notag \\
        &= p(p-1) \int_\Omega  \nabla \psi (n) \cdot \nabla c \notag \\
        &=-p(p-1) \int_\Omega \psi(n ) \Delta c \notag\\
        &=p(p-1) \int_\Omega \psi(n)n - p(p-1) \int_\Omega \psi (n) c - \int_\Omega \psi (n) u \cdot \nabla c,
    \end{align}
    where 
    \begin{align*}
        \psi (s)= \int_0^s z^{p-1}S(z) \, dz .
    \end{align*}
    Thanks to \eqref{S}, we can estimate
    \begin{align} \label{L5.1'}
        |\psi (s)|&\leq K_S \int_0^s z^{p-1}(1+z)^{-\alpha}\, dz \notag \\
        &\leq K_S \int_0^s  z^{p-\alpha-1}\, dz \notag \\
        &\leq \frac{K_S}{p-\alpha} s^{p-\alpha}.
    \end{align}
    By using divergence free condition of $u$ and Young's inequality, there exists $c_1>0$ and $c_2>0$ such that 
    \begin{align} \label{L5.2}
        - \int_\Omega \psi (n) u \cdot \nabla c &= \int_\Omega c \nabla \psi (n) \cdot u \notag \\
        &\leq c_1 \int_\Omega n^{p-\alpha-1} |\nabla n||u|c \notag \\
        &\leq \frac{p(p-1)}{2} \int_\Omega n^{p-2}|\nabla n|^2 +c_2 \int_\Omega n^{p-2\alpha} |u|^2c^2.
    \end{align}
    Employing Hölder's inequality and Lemma \ref{L2}, we deduce that 
    \begin{align} \label{L5.3}
        c_2 \int_\Omega n^{p-2\alpha} |u|^2c^2 &\leq c_2 \left ( \int_\Omega |u|^{\frac{2q}{q-1}} \right )^\frac{q-1}{q} \left ( \int_\Omega n^{q(p-2\alpha)}c^{2q} \right )^\frac{1}{q} \notag \\
        &\leq c_3\left ( \int_\Omega n^{q(p-2\alpha)}c^{2q} \right )^\frac{1}{q}, 
    \end{align}
    where $q>3$ and $c_3>0$. Now, we apply Young's inequality and Lemma \ref{L4} to obtain that 
    \begin{align} \label{L5.4}
        \int_\Omega n^{q(p-2\alpha)}c^{2q} &\leq \int_\Omega n^{q(p-2\alpha) \lambda} + \int_\Omega c^{\frac{2q\lambda}{\lambda-1}} \notag \\
        &\leq \int_\Omega n^{q(p-2\alpha) \lambda} + c_4\int_\Omega n^{\frac{2q \lambda}{r(\lambda-1)}} +c_5 ,
    \end{align}
    for some $\lambda>1$, $1<r<3$, $c_4>0$, and $c_5>0$. By fixing
     \begin{align}\label{L5.4'}
        r= \frac{9}{2+3\alpha},\qquad \text{and }\lambda = 1+\frac{2}{r(p-2\alpha)},
    \end{align}
    then we find that 
    \begin{align*}
        q(p-2\alpha) \lambda = \frac{2q\lambda}{\lambda-1}= q(p-2\alpha)+\frac{2q}{r}.
    \end{align*}
    Substituting this into \eqref{L5.4} and \eqref{L5.3}, we arrive at 
    \begin{align*}
          c_2 \int_\Omega n^{p-2\alpha} |u|^2c^2 &\leq c_3 \left ( (c_4+1) \int_\Omega n^{q(p-2\alpha)+\frac{2q}{r}} +c_5 \right )^{\frac{1}{q}} \notag \\
          &\leq c_6 \left (\int_\Omega n^{q(p-2\alpha)+\frac{2q}{r}} \right )^\frac{1}{q} +c_7,
    \end{align*}
    where $c_6=c_3(c_4+1)^{\frac{1}{q}}$ and $c_7=c_3c_5^\frac{1}{q}$. Choosing 
    \begin{equation}\label{L5.4''}
        3<q< \min \left \{ \frac{3}{(2-3\alpha)_+} , \frac{3p}{p-2\alpha+\frac{2}{r}} \right \}
    \end{equation}
   allowing us to apply Gagliardo–Nirenberg interpolation inequality, we can find positive constants $c_8$, $c_9$ and $c_{10}$ such that 
    \begin{align}\label{L5.5}
         c_6 \left (\int_\Omega n^{q(p-2\alpha)+\frac{2q}{r}} \right )^\frac{1}{q} &\leq c_8 \left ( \int_\Omega n^{p-2} |\nabla n|^2 \right )^{\frac{q(p-2\alpha)+\frac{2q}{r}-1}{pq -\frac{q}{3}}} \left ( \int_\Omega n \right )^{\frac{1- \frac{q(p-2\alpha)+\frac{2q}{r}}{3p}}{pq-\frac{q}{3}}} \notag \\
         &\quad+ c_8 \left ( \int_\Omega n \right )^{p-2\alpha+ \frac{2}{r}} \notag \\
         &\leq c_{10}\left ( \int_\Omega n^{p-2} |\nabla n|^2 \right )^{\frac{q(p-2\alpha)+\frac{2q}{r}-1}{pq -\frac{q}{3}}} +c_{10},
    \end{align}
    where we used Lemma $\ref{L1}$ in the last inequality. Noting that 
\begin{align}\label{L5.6}
    {\frac{q(p-2\alpha)+\frac{2q}{r}-1}{pq -\frac{q}{3}}} <1
\end{align}
when 
\begin{align*}
    \alpha > \frac{1}{6} + \frac{1}{r} - \frac{1}{2q},
\end{align*}  
which is true since from \eqref{L5.4'} and \eqref{L5.4''} we have
\begin{align*}
    \frac{1}{r} -\frac{1}{3} = \frac{1}{3} \left ( \alpha-\frac{1}{3} \right )
\end{align*}
and 
\begin{align*}
    \frac{1}{6}-\frac{1}{2q} < \frac{1}{2} \left ( \alpha- \frac{1}{3} \right ).
\end{align*}
Thanks to \eqref{L5.6}, we can apply Young's inequality to \eqref{L5.5} and then substitute it into \eqref{L5.3} to obtain
\begin{align}\label{L5.7}
   c_2 \int_\Omega n^{p-2\alpha} |u|^2c^2 \leq \frac{p(p-1)}{4} \int_\Omega n^{p-2}|\nabla n|^2 +c_{11},
\end{align}
where $c_{11}>0$. Using \eqref{L5.1'}, Young's inequality and Lemma \ref{L4}, we can estimate 
\begin{align}\label{L5.8}
   \int_\Omega n^p+ p(p-1) \int_\Omega \psi(n)n -p(p-1)\int_\Omega \psi(n)c &\leq\int_\Omega n^p+ c_{12} \int_\Omega n^{p-\alpha+1} + c_{12} \int_\Omega n^{p-\alpha}c \notag \\
    &\leq \int_\Omega n^p+ 2c_{12} \int_\Omega n^{p-\alpha+1}+ + c_{12} \int_\Omega c^{p-\alpha+1} \notag \\
    &\leq c_{13} \int_\Omega n^{p-\alpha+1} +c_{14}.
\end{align}
We apply Gagliardo–Nirenberg interpolation inequality and Lemma \ref{L1}, to find $c_{15}>0$ and $c_{16}>0$ such that
\begin{align*}
    c_{13}\int_\Omega n^{p-\alpha+1} &\leq  c_{15} \left ( \int_\Omega n^{p-2}|\nabla n|^2 \right ) ^{\frac{3p-3\alpha}{3p-1}} \left ( \int_\Omega n \right )^{\frac{2p+\alpha-1}{3p-1}} +c_{15} \left ( \int_\Omega n \right )^{p-\alpha+1} \notag \\
    &\leq c_{16}\left ( \int_\Omega n^{p-2}|\nabla n|^2 \right ) ^{\frac{3p-3\alpha}{3p-1}} +c_{16}.
\end{align*}
Since $\frac{3p-3\alpha}{3p-1}<1$ when $\alpha>\frac{1}{3}$, we can employ Young's inequality to this to find $c_{17}>0$ such that
\begin{align}\label{L5.9}
    c_{13}\int_\Omega n^{p-\alpha+1} &\leq  \frac{p(p-1)}{4} \int_\Omega n^{p-2}|\nabla n|^2 + c_{17}.
\end{align}
Collecting \eqref{L5.1}, \eqref{L5.7}, \eqref{L5.8}, and \eqref{L5.9}, we arrive at 
\begin{align*}
    \frac{d}{dt}\int_\Omega n^p +\int_\Omega n^p \leq c_{18} \qquad \text{for all }t\in (0,T_{\rm max}),
\end{align*}
for some $c_{18}>0$. Finally, we apply Grönwall's inequality to complete the proof.

\end{proof}

\subsection{Proof of the main result for $N=3$} \label{S4-2}
With the aid of Lemma~\ref{L5}, the regularity of $u$ can be further improved by employing standard arguments based on the regularity properties of the Stokes semigroup.

\begin{lemma} \label{ubdd}
    There exists $C>0$ such that 
    \begin{align}\label{ubdd-1}
        \left \|A^\beta u(\cdot,t) \right \|_{L^2(\Omega)} \leq  C \qquad \text{for all }t\in (0,T_{\rm max}),
    \end{align}
    and 
    \begin{align}\label{ubdd-2}
        \left \| u(\cdot,t) \right \|_{L^\infty(\Omega)} \leq  C \qquad \text{for all }t\in (0,T_{\rm max}).
    \end{align}
\end{lemma} 

\begin{proof}
    By using the fractional power $A^\beta $ to the Duhamel formula with Stokes subsystem of \eqref{1}
    \begin{align*}
        u(\cdot,t)= e^{-tA}u_0 + \int_0^t e^{-(t-\tau )A} \mathcal{P}(n \nabla \phi )(\cdot, \tau )\, d\tau \qquad \text{for }t\in (0,T_{\rm max}),
    \end{align*}
    we obtain 
    \begin{align*}
        \left \| A^\beta u(\cdot,t) \right \|_{L^2(\Omega)} \leq \left \| e^{-tA}A^\beta u_0 \right \|_{L^2(\Omega)} + \int_0^t \left \| A^\beta e^{-(t-\tau )A} \mathcal{P}(n \nabla \phi )(\cdot, \tau ) \right \|_{L^2(\Omega)} \, d\tau.
    \end{align*}
    Thanks to the smoothing properties of the Stokes semigroup \cite{Sohr2001} and Lemma \ref{L5}, we can find $c_1>0$ and $\lambda_1>0$ such that 
    \begin{align}
         \left \| A^\beta u(\cdot,t) \right \|_{L^2(\Omega)} &\leq \left \|A^\beta u_0 \right \|_{L^2(\Omega)} + c_1\int_0^t (t-s)^{-\beta }e^{-\lambda_1(t-s)} \left \| \mathcal{P}(n \nabla \phi )(\cdot, \tau ) \right \|_{L^2(\Omega)} \, d\tau \notag \\
         &\leq \left \|A^\beta u_0 \right \|_{L^2(\Omega)} + c_1 \left \| \nabla \phi  \right \|_{L^\infty(\Omega)} \cdot \sup_{t\in (0,T_{\rm max})} \left \{ \left \|n(\cdot ,t) \right \|_{L^2(\Omega)} \right \}  \cdot \int_0^\infty s^{-\beta }e^{-\lambda_1 s} \, ds
    \end{align}
    which implies \eqref{ubdd-1} since $u_0 \in D(A^\beta )$ and $\int_0^\infty s^{-\beta }e^{-\lambda_1 s} \, ds <\infty$. As the assumption $\beta \in \left ( \frac{3}{4},1 \right )$ ensures that $D(A^\beta) \to L^\infty (\Omega; \mathbb{R}^3 )$ \cite{Giga1981,Hen}, also entails \eqref{ubdd-2}. The proof is now complete.
\end{proof}

By utilizing the similar arguments in the proof of Lemma \ref{cbdd'}, we can establish the uniform boundedness of $\nabla c$ as follows:

\begin{lemma}\label{cbdd}
    There exists $C>0$ such that 
    \begin{align} 
         \left \| c(\cdot,t) \right \|_{W^{1,\infty}(\Omega)} \leq  C \qquad \text{for all }t\in (0,T_{\rm max}).
    \end{align}
\end{lemma}
\begin{proof}
    Testing the second equation of \eqref{1} by $c$, using Lemma \ref{L5} and applying Young's inequality yields 
    \begin{align*}
        \int_\Omega |\nabla c|^2 +\int_\Omega c^2 &= \int_\Omega nc \notag \\
        &\leq \frac{1}{2}\int_\Omega n^2 + \frac{1}{2} \int_\Omega c^2,
    \end{align*}
    which implies that 
    \begin{align} \label{cbdd.1}
        \left \| c(\cdot,t) \right \|_{W^{1,2}(\Omega)} \leq c_1 \qquad \text{for all } t\in (0,T_{\rm max}),
    \end{align}
    for some $c_1>0$. Owing to Lemma \ref{L5} with $p=4$ and the uniform boundedness of $u$ in Lemma \ref{ubdd}, there exist $c_2>0$ and $c_3>0$ such that 
    \begin{align*}
         \left \| n(\cdot,t) \right \|_{L^4(\Omega)} \leq c_2 \qquad \text{for all }t\in (0,T_{\rm max}),
    \end{align*}
    and 
    \begin{align*}
         \left \| u(\cdot,t) \right \|_{L^\infty(\Omega)} \leq c_3 \qquad \text{for all }t\in (0,T_{\rm max}).
    \end{align*}
    We apply $L^p$ elliptic regularity theory, \eqref{cbdd.1}, and Gagliardo–Nirenberg interpolation inequality to find positive constants  $c_4, c_5$ and $c_6$ such that 
    \begin{align*} 
        \left \| c \right \|_{W^{2,4}(\Omega)} &\leq c_4 \left \| n \right \|_{L^4(\Omega)} +c_4 \left \| u \cdot \nabla c  \right \|_{L^4(\Omega)}  \notag \\
        &\leq c_2c_4 +c_3c_4  \left \| \nabla c  \right \|_{L^4(\Omega)} \notag \\
        &\leq c_2c_4 +c_3c_4c_5  \left \| c \right \|_{W^{2,4}(\Omega)}^{\frac{3}{7}} \left \| \nabla c  \right \|^{\frac{4}{7}}_{L^2(\Omega)} \notag \\
        &\leq \frac{1}{2} \left \| c \right \|_{W^{2,4}(\Omega)}  +c_6,
    \end{align*}
    which implies that 
    \begin{align*} 
        \left \| c(\cdot,t) \right \|_{W^{2,4}(\Omega)} \leq 2c_6 \qquad \text{for all } t\in (0,T_{\rm max}).
    \end{align*}
    This, together with the Sobolev embedding $W^{2,4}(\Omega) \to W^{1, \infty}(\Omega)$ yields 
     \begin{align*} 
        \left \| c(\cdot,t) \right \|_{W^{1,\infty}(\Omega)} \leq c_7 \left \| c(\cdot,t) \right \|_{W^{2,4}(\Omega)} \leq 2c_7c_6 \qquad \text{for all } t\in (0,T_{\rm max}).
    \end{align*}
    for some $c_7>0$. The proof is now complete.
    
\end{proof}

Equipped with the boundedness of $\nabla c$ established in the previous lemma, we can now apply a standard Moser iteration argument to prove the uniform boundedness of $n$ as follows.

\begin{lemma} \label{nbdd}
    There exists $C>0$ such that 
    \begin{align} \label{nbdd.1}
         \left \| n(\cdot,t) \right \|_{L^{\infty}(\Omega)} \leq  C \qquad \text{for all }t\in (0,T_{\rm max}).
    \end{align}
\end{lemma}
\begin{proof}
    Multiplying the first equation of \eqref{1} by $pn^{p-1}$, applying Lemma \ref{cbdd}, Young's inequality and noting that $S(n ) \leq K_S$, we can find $c_1>0$ such that 
    \begin{align*}
        \frac{d}{dt}\int_\Omega n^p &=-p(p-1) \int_\Omega n^{p-2}|\nabla n|^2 +p(p-1) \int_\Omega n^{p-1}S(n) \nabla n \cdot \nabla c \notag \\
        &\leq -\frac{p(p-1)}{2}) \int_\Omega n^{p-2}|\nabla n|^2 + \frac{p(p-1)}{2}\int_\Omega n^p S^2(n)|\nabla c|^2 \notag \\
        &\leq -\frac{p(p-1)}{2}) \int_\Omega n^{p-2}|\nabla n|^2  +c_1 p(p-1)\int_\Omega n^p \qquad \text{for all } t\in (0,T_{\rm max}). 
    \end{align*}
    Since $n \in L^\infty \left ((0,T_{\rm max});L^p(\Omega) \right )$ for any $p>1$, we can use Moser-type iterative argument (see  \cite{Alikakos1}, \cite{Alikakos2}) or \cite{Winkler-2011}[Lemma A.1]) to prove \eqref{nbdd.1}. The proof is now complete.
\end{proof}

We are now in position to prove the second main result.
\begin{proof}[Proof of Theorem \ref{thm2}]
    By combining Lemmas \ref{nbdd}, \ref{cbdd}, \ref{ubdd} and \ref{Local}[\eqref{Local.1}], it follows that $T_{\rm max}= \infty$ and 
    \begin{align*}
        \sup_{t>0} \left \{\left \| n(\cdot,t) \right \|_{L^\infty(\Omega)} +    \left \| c(\cdot,t) \right \|_{W^{1,\infty}(\Omega)} +  \left \| u(\cdot,t) \right \|_{L^\infty(\Omega)}  \right \} < \infty,
    \end{align*}
    which implies \eqref{thm2-1} and thereby concludes the proof.
\end{proof}

\paragraph{Acknowledgments} Minh Le was supported by the Hangzhou Postdoctoral Research Grant (No.~207010136582503).

     \paragraph{Data Availability}
 Data sharing not applicable to this article as no datasets were generated or analyzed during
the current study.

\section*{Conflict of Interest}

The author declares that there is no conflict of interest.

\end{document}